\documentclass[twoside,leqno]{article}

\usepackage[letterpaper]{geometry}

\usepackage{siamproceedings}

\usepackage[T1]{fontenc}
\usepackage{amsfonts}
\usepackage{graphicx}
\usepackage{epstopdf}
\usepackage{enumitem}
\usepackage{algorithmic}
\ifpdf
  \DeclareGraphicsExtensions{.eps,.pdf,.png,.jpg}
\else
  \DeclareGraphicsExtensions{.eps}
\fi

\newsiamthm{thm}{Theorem}
\newsiamthm{conj}{Conjecture}
\newsiamthm{lem}{Lemma}
\newsiamthm{problem}{Problem}
\newsiamthm{postulate}{Postulate}

\newsiamremark{remark}{Remark}
\newsiamremark{hypothesis}{Hypothesis}
\crefname{hypothesis}{Hypothesis}{Hypotheses}
\newsiamthm{claim}{Claim}

\usepackage{amsopn}

\begin{document}

\newcommand\relatedversion{}

\title{\Large Refined Absorption: A New Proof of the Existence Conjecture and its Applications to Extremal and Probabilistic Design Theory\relatedversion}
    \author{Luke Postle\thanks{University of Waterloo (\email{lpostle@math.uwaterloo.ca}).}
}

\date{}

\maketitle

% Copyright Statement
% When submitting your final paper to a SIAM proceedings, it is requested that you include
% the appropriate copyright in the footer of the paper.  The copyright added should be
% consistent with the copyright selected on the copyright form submitted with the paper.
% Please note that "20XX" should be changed to the year of the meeting.

% Default Copyright Statement
\fancyfoot[R]{\scriptsize{Copyright \textcopyright\ 20XX by SIAM\\
Unauthorized reproduction of this article is prohibited}}

% Depending on which copyright you agree to when you sign the copyright form, the copyright
% can be changed to one of the following after commenting out the default copyright statement
% above.

%\fancyfoot[R]{\scriptsize{Copyright \textcopyright\ 20XX\\
%Copyright for this paper is retained by authors}}

%\fancyfoot[R]{\scriptsize{Copyright \textcopyright\ 20XX\\
%Copyright retained by principal author's organization}}

%\pagenumbering{arabic}
%\setcounter{page}{1}%Leave this line commented out.

\begin{abstract} 
We discuss the recently developed method of refined absorption and how it is used to provide a new proof of the Existence Conjecture for combinatorial designs. This method can also be applied to resolve open problems in extremal and probabilistic design theory while providing a unified framework for these problems. Crucially, the main absorption theorem can be used as a ``black-box'' in these applications obviating the need to reprove the absorption step for each different setup.
\end{abstract}

\section{Introduction.}

\subsection{Design Theory.}

The study of combinatorial designs has a rich history spanning nearly two centuries. One of the most classical theorems in all of design theory is a result of Kirkman~\cite{K47} which classifies for which $n$ there exists a set of triples of an $n$-set $X$ such that every pair in $X$ is in exactly one triple; such a set is called a \emph{Steiner Triple System}. From the graph-theoretic perspective, this is equivalent to a decomposition of the edges of the complete graph $K_n$ into edge-disjoint triangles. Kirkman proved that the necessary divisibility conditions for this, namely that $3~|~\binom{n}{2}$ and $2~|~(n-1)$ which equates to $n\equiv 1,3 \mod 6$, are also sufficient. 

Arguably the most studied object in design theory is a natural generalization of this object as follows. Given integers $n > q > r \geq 1$, an $(n,q,r)$\textit{-Steiner System} is a set $S$ of $q$-subsets of an $n$-set $X$ such that each $r$-subset of $X$ is contained in exactly one element of $S$. From the hypergraph-theoretic perspective, an $(n,q,r)$-Steiner system is equivalent to a decomposition of the edges of $K_n^r$ (the complete $r$-uniform hypergraph on $n$ vertices) into edge-disjoint copies of $K_q^r$ (referred to as cliques). More generally, a \emph{design} with parameters $(n,q,r,\lambda)$ is a set $S$ of $q$-subsets of an $n$-set $X$ such that every $r$-subset of $X$ belongs to exactly $\lambda$ elements of $S$. From the hypergraph-theoretic perspective, an $(n,q,r,\lambda)$-design is equivalent to a decomposition of the multigraph $\lambda*K_q^r$ (meaning each $r$-set has $\lambda$ parallel edges) into a set of edge-disjoint cliques.  Once again, there are necessary divisibility conditions for the existence of an $(n, q, r, \lambda)$-design: for each $0 \leq i \leq r-1,$ we require $\binom{q-i}{r-i}~|~~\lambda\cdot \binom{n-i}{r-i}$.
Whether these divisibility conditions suffice for all large enough $n$ became a notorious folklore conjecture from the 1800s called the Existence of Combinatorial Designs Conjecture.

\begin{conj}[Existence Conjecture]\label{conj:Existence}
Let $q > r \ge 2$ and $\lambda \geq 1$ be integers. If $n$ is sufficiently large and $\binom{q-i}{r-i}~|~\lambda \cdot \binom{n-i}{r-i}$ for all $0\le i \le r-1$, then there exists a design with parameters $(n,q,r,\lambda)$.
\end{conj}

\subsection{History of the Existence Conjecture.}

In 1847, Kirkman~\cite{K47} proved this when $q=3$, $r=2$ and $\lambda=1$. In the 1970s Wilson \cite{WI, WII, WIII} revolutionized design theory when he proved the `graph' case of the Existence Conjecture, namely for all values of $q$ when $r = 2$. In 1963, Erd\H{o}s and Hanani~\cite{EH63} conjectured an {\em approximate} version of the Existence Conjecture, positing that it is possible to find a packing of edge-disjoint copies of $K_q^r$ covering $(1-o(1))$ proportion of the edges of $K_n^r$. This problem was solved in 1985 by R\"{o}dl~\cite{R85} using the celebrated ``nibble'' method. Only in the last decade was the Existence Conjecture fully resolved.

\begin{thm}[Keevash~\cite{K14}]\label{thm:Existence}
Conjecture~\ref{conj:Existence} is true.   
\end{thm}

Namely in 2014, Keevash~\cite{K14} proved the Existence Conjecture using \emph{randomized algebraic constructions}. Thereafter in 2016, Glock, K\"{u}hn, Lo, and Osthus~\cite{GKLO16} gave a purely combinatorial proof of the Existence Conjecture via \emph{iterative absorption}. Both approaches have different benefits and each led to subsequent work over the years using these approaches. In 2024, Delcourt and Postle~\cite{DPI} developed the methodology of {\em refined absorption}, giving a one-step, combinatorial proof of the Existence Conjecture. Later in 2024, Keevash~\cite{K24} provided a more concise proof that also uses the new refined absorption framework. The main goal of this article is to explain this new framework and its applications to problems in extremal and probabilistic design theory.

\subsection{Hypergraph and Graph Decompositions.} While at first glance these design theory questions seem set-theoretic, it is useful to recast them from a graph-theoretic perspective. Given hypergraphs $F$ and $G$, an \emph{$F$-decomposition} of $G$ is a partition of the edges of $G$ into copies of $F$ while an \emph{$F$-packing} is a set of pairwise edge-disjoint copies of $F$. For brevity, we write \emph{$r$-graph} for $r$-uniform hypergraph. As noted, we let $K_n^r$ denote the complete $r$-graph on $n$ vertices. As mentioned, an $(n,q,r)$-Steiner system is equivalent to a $K_q^r$-decomposition of $K_n^r$. There are obvious necessary divisibility conditions for an $r$-graph $G$ to admit a $K_q^r$-decomposition. To that end, an $r$-graph $G$ is \emph{$K_q^r$-divisible} if for all $i\in \{0,1,\ldots, r-1\}$, we have that $\binom{q-i}{r-i}~|~|\{e\in G: S\subseteq e\}|$ for all subsets $S$ of $V(G)$ of size $i$.

\subsection{Nibble and Absorbers.}
There are two crucial concepts used in all the known proofs of the Existence Conjecture, \emph{nibble} and \emph{absorbers}; nibble allows the construction of an almost full Steiner system while absorbers allows us to complete such partial constructions to a full Steiner system. Let us discuss these now in more detail.

The idea of nibble is that instead of finding a desired object by a simple application of the probabilistic method, we instead build the object randomly a little bit at a time while carefully controlling the properties of future sampling. (The first appearance of this idea dates to the early 1980s in the so-called \emph{semi-random method}, see the work of Ajtai, Koml\'os, Szemer\'edi~\cite{AKS81} from 1981 for finding a large independent set in a triangle-free graph). In 1985, R\"odl proved the Erd\H{o}s-Hanani conjecture using a similar idea by building a partial Steiner system one little `nibble' at a time. R\"odl's work was generalized over the ensuing decades: first to hypergraph matchings by Frankl and R\"odl~\cite{FR85} in 1985, then to edge-coloring hypergraphs by Pippenger and Spencer~\cite{PS89} in 1989, and then to list-edge-coloring hypergraphs by Kahn~\cite{K96} in 1996. See the survey of Kang, Kelly, K\"uhn, Methuku, and Osthus for more history of nibble~\cite{KKKMO21}. Here is the nibble theorem we require for Existence specialized to the setting of designs. Note for an $r$-graph $G$, we let $K_q^r(G)$ denote the set of $K_q^r$-cliques of $G$.

\begin{theorem}[Nibble for Designs]\label{thm:NibbleDesigns}
$\forall~q > r\ge 1, \beta \in (0,1)$, $\exists~\alpha\in (0,1)$ s.t. for all large enough $D$:
\vskip.05in
\noindent If $G\subseteq K_n^r$ and $\mathcal{H}\subseteq K_q^r(G)$ s.t.~every edge of $G$ is in $D \cdot \left(1 \pm D^{-\beta}\right)$ cliques of $\mathcal{H}$ while every pair of edges of $G$ is in at most $D^{1-\beta}$ cliques of $\mathcal{H}$, then $\exists$~a $K_q^r$-packing of $G$ of size at least $\frac{e(G)}{\binom{q}{r}} \cdot \left(1-D^{-\alpha}\right)$. 
\end{theorem}

If we apply the above theorem in the case $D:= \rho\cdot \binom{n}{q-r}$ where $\rho$ is some fixed constant in $(0,1]$, then the statement may be simplified further as follows since every pair of edges is in at most $\binom{n}{q-r-1} = \Theta(n^{q-r-1}) \le n^{q-r-1/2}$ if $n$ is large enough.

\begin{corollary}\label{cor:NibbleDesignsSpecific}
$\forall~q > r\ge 1,~\rho \in (0,1]$, $\exists~~\alpha\in (0,1)$ s.t. for all large enough $n$: If $G\subseteq K_n^r$ and $\mathcal{H}\subseteq K_q^r(G)$ s.t.~every edge of $G$ is in $\left(\rho \pm n^{-1/3}\right) \binom{n}{q-r}$ cliques of $\mathcal{H}$, then $\exists~$a $K_q^r$-packing of $G$ of size at least $\frac{e(G)}{\binom{q}{r}} \cdot \left(1-n^{-\alpha}\right)$. 
\end{corollary}

The idea of absorbers is that we can complete an almost object into a full object by reserving a magical structure called an `absorber' that can absorb whatever is not in the almost object. While its origins can be traced to earlier works in the 1980s and 1990s, it has been said that R\"odl, Ruci\'nski and Szemer\'edi~\cite{RRS06} in 2006 popularized and codified the method in their work on finding perfect matchings in high minimum degree hypergraphs. Since then, the method has exploded in popularity being used for finding Hamilton cycle decompositions, decompositions into trees, designs, rainbow matchings, transversals in Latin squares and many more. 

We now recall the relevant absorber definitions for $K_q^r$-decompositions. First we need the definition of an \emph{absorber} for a specific leftover $L$ of nibble.

\begin{definition}[Absorber] Let $L$ be a $K_q^r$-divisible $r$-graph.  An $r$-graph $A$ is a \emph{$K_q^r$-absorber} for $L$ if $V(L)$ is independent in $A$ and both $A$ and $A\cup L$ admit $K_q^r$-decompositions.    
\end{definition}

More generally, we require an \emph{omni-absorber} that can absorb any (divisible) leftover $L$ located inside a subgraph $X$ as follows.

\begin{definition}[Omni-Absorber]\label{def:OmniAbsorbers}
An $r$-graph $A$ is a \emph{$K_q^r$-omni-absorber} for an $r$-graph $X$ with \emph{decomposition family} $\mathcal{F}_A$ and \emph{decomposition function} $\mathcal{Q}_A$  if $X$ and $A$ are edge-disjoint and for every $K_q^r$-divisible subgraph $L$ of $X$, $\mathcal{Q}_A(L)\subseteq \mathcal{F}_A$ is a $K_q^r$-decomposition of $A\cup L$. 
\end{definition}

While the existence of absorbers follows from high minimum degree generalizations of the Existence Conjecture, their existence can be proved in their own right and thus be used toward a proof of Existence as follows.

\begin{theorem}\label{thm:AbsorberExistence}  $\forall~q > r\ge 1$: If $L$ is a $K_q^r$-divisible $r$-graph, then there exists a $K_q^r$-absorber for $L$.
\end{theorem}

Barber, K\"uhn, Lo, and Osthus~\cite{BKLO16} proved the above theorem for $r=2$ (namely in the graph case) for which they gave an explicit construction. A proof of Theorem~\ref{thm:AbsorberExistence} can be extracted from the second proof of Existence by Glock, K\"uhn, Lo, and Osthus~\cite{GKLO16} in about 10 pages (assuming that the Existence Conjecture holds for smaller uniformities); recently, Delcourt, Kelly, and Postle~\cite{DKP24} provided a short self-contained proof (around 2 pages) inspired by Keevash's original work~\cite{K14} which we discuss in more detail later. 

We note then that the existence of omni-absorbers follows from the existence of absorbers (provided $X$ is confined to a small subset of the vertices) by taking the union of $K_q^r$-absorbers (vertex-disjoint outside of $X$) for each $K_q^r$-divisible subgraph $L$ of $X$. However, this construction is wildly inefficient (requiring $v(X)$ to be at most logarithmic in $n$). The development of polynomially efficient and even linearly efficient omni-absorbers is one of the central pillars of the new refined absorption framework as we next discuss.

\subsection{Further Ideas in the Proofs of Existence.} Armed with the two spectacular ideas of \emph{nibble} and \emph{absorbers}, how should one proceed to prove the Existence Conjecture? Nibble will find a partial $(n,q,r)$-Steiner system using almost all the edges of $K_n^r$, however there will still remain a set of \emph{leftover} edges. Inspired by absorption, it is natural to dream that we could find a magical `absorbing' structure that will be able to decompose this leftover. Indeed, Keevash's original proof~\cite{K14} may be interpreted as building an omni-absorber in an algebraic manner for one specifically constructed $X$ that meets certain algebraic conditions. 

On the other hand, the second proof of the Existence Conjecture by Glock, K\"{u}hn, Lo, and Osthus~\cite{GKLO16}, as well as many other recent breakthroughs in design theory and hypergraph decompositions, proceeds via \emph{iterative absorption}, in which the location of potential leftovers of a decomposition problem are iteratively reduced until for one final (small) location of the leftover we finally absorb via an omni-absorber set aside at the beginning. The method of iterative absorption was first introduced by Knox, K\"{u}hn, and Osthus in~\cite{KKO15} and first used for decompositions by K\"{u}hn and Osthus in~\cite{KO13}. 

There are two natural approaches for iterative absorption in designs. The first (arguably more natural to the problem) is to iteratively reduce the location via nested sets of edges $E_0\supset E_1 \supset E_2 \ldots \supset E_k= X$ (called an \emph{edge vortex}); the second approach is via nested sets of vertices $V_0\supset V_1 \supset V_2\ldots \supset V_k =X$ (called a \emph{vertex vortex}). The issue with the edge vortex approach is that there is a limit to how small the edge sets can be; namely if $E_i$ is constructed randomly with probability $p_i$, one would need $p_i$ large enough (even for triangles it would need to be at least $n^{-1/2}$) to guarantee that $E_i$ contains $K_q^r$'s. The benefit of a vertex vortex meanwhile is that $v(X)$ may be small, even constant sized, and hence the use of inefficient omni-absorbers is permitted. The cost is that the remaining edges from $V_i\setminus V_{i+1}$ to $V_{i+1}$ must be decomposed and these edges must be decomposed via $K_{q-j}^{r-j}$ factors inside $V_{i+1}$ where $j$ is the number of vertices of the edge in $V_i\setminus V_{i+1}$; this is the technical heart of the Cover Down Lemma of Glock, K\"{u}hn, Lo, and Osthus~\cite{GKLO16}. One issue with the cover down approach is that it does not produce a black-box theorem (as in Theorem~\ref{thm:RefinedEfficientOmniAbsorber} below) and so the whole proof and method must be modified for various variants. 

The first approach of an edge vortex had not been generally used for designs due to only having access to inefficient omni-absorbers. One of the main breakthroughs of refined absorption is the ability to construct extremely efficient omni-absorbers as follows. Note that we also need our omni-absorber $A$ to have low maximum degree if we hope to apply nibble to $K_n^r\setminus (X\cup A)$. We also note that Keevash's new alternate proof of efficient omni-absorbers~\cite{K14} does not yield linear efficiency as in the following theorem but only a certain polynomial proportionality (which is sufficient for a proof of Existence). %Note that for an $r$-graph $G$, $\Delta_{r-1}(G)$ (resp.~$\delta_{r-1}(G)$) denotes the \emph{maximum $(r-1)$-degree} (resp.~minimum) of $G$. 

\begin{thm}[Efficient Omni-Absorber~\cite{DPI}]\label{thm:EfficientOmniAbsorber}
$\forall~q > r \ge 1$, $\exists~C\ge 1$: If $X\subseteq  K_n^r$ and $\Delta_{r-1}(X) \le \frac{n}{C}$, then there exists a $K_q^r$-omni-absorber $A$ for $X$ with $\Delta_{r-1}(A) \le C\cdot \max\left\{\Delta_{r-1}(X),~n^{1-1/r}\cdot \log n\right\}$.    
\end{thm}

Aided now by the efficient omni-absorbers of Theorem~\ref{thm:EfficientOmniAbsorber}, we will use the edge vortex approach to provide the refined absorption proof of the Existence Conjecture which we overview in the next section. Since we use an edge vortex instead of a vertex vortex, the cover down step can be replaced by a `nibble with reserves' step. Furthermore, since our omni-absorbers are so efficient, we only use one step; that is, we only need to choose one random `reserve' subset $X=E_1$ of $E_0=E(K_n^r)$. In summary, iterative absorption iteratively \emph{covers down} to $X$ the location of a leftover via random processes; in contrast, our efficient omni-absorbers are themselves constructed via a vertex vortex of sorts; but in contrast, this can be viewed as iteratively \emph{refining down} the omni-absorber (and its \emph{remainder}, the part of the omni-absorber for which an absorption plan is not yet complete) via more structural methods.

That said, even assuming the existence of our magical omni-absorber, there is one last technical hurdle to overcome in the new proof of Existence. The number of cliques of $K_n^r\setminus (X\cup A)$ containing an edge may no longer be the same for every edge (as it was in $K_n^r$). Worse, the size of the initial leftover of nibble is proportional to this irregularity. Since we do not attempt to control the building of the absorber $A$, that irregularity would be proportional to the size of $A$ which is proportional to the size of the reserves $X$; but $X$ needs to be large enough to decompose that initial leftover! This creates a circuit of dependencies, a \emph{hierarchy problem}. Luckily, this problem can be overcome by a ``regularity boosting'' step wherein one cleverly chooses a subset of the cliques in $K_n^r\setminus (A\cup X)$ that is much more regular than the set of all cliques and only run nibble using this special set of cliques. Glock, K\"uhn, Lo, and Osthus~\cite{GKLO16} developed machinery to solve this problem; essentially, they showed with a short ingenious argument that $r$-graphs with minimum $(r-1)$-degree close to $n$ (such as $K_n^r\setminus (X\cup A)$ in our application above) admit low-weight fractional $K_q^r$-decompositions and hence the resulting distribution can be sampled from to be very regular. We discuss this further in the next section as well.

\subsection{Benefits of Refined Absorption and its Black-Box Nature.}

Before moving on to the refined absorption proof of the Existence Conjecture, let us spare some words for the method's applications to other problems as well as its benefits compared to the other approaches. 

While the Existence Conjecture took nearly two centuries to prove, there are still many mountains left to climb -- variants and generalizations of the Existence Conjecture still left to consider. Indeed, as often is the case in mathematics, it is only when we reach the summit of one mountain that we can truly see all the mountains that lay beyond.  

There are three main problems we will consider in this article. The first two are central problems in what may be called \emph{extremal design theory}, an area which inherits from extremal graph theory and generalizes its questions to designs. The third problem is central to what may be called \emph{probabilistic design theory} which inherits from probabilistic graph theory and generalizes its questions to designs. 

\begin{itemize}
    \item \textbf{Minimum Degree Thresholds for $K_q^r$-decompositions}: What is the minimum $(r-1)$-degree that forces a $K_q^r$-divisible $r$-graph on $n$ vertices to admit a $K_q^r$-decomposition?
    \item \textbf{Existence of High Girth Steiner Systems}: Do there exist $(n,q,r)$-Steiner systems of large girth for all $q > r$ and $n$ large enough?
    \item \textbf{Probabilistic Thresholds for Steiner Systems}: What is the minimum probability $p$ such that the binomial random $q$-graph $\mathcal{G}^{(q)}(n,p)$ (i.e.~$q$-sets are chosen independently with probability $p$) contains an $(n,q,r)$-Steiner system asymptotically almost surely?
\end{itemize}

Now Theorem~\ref{thm:EfficientOmniAbsorber} is not sufficient for these applications, at least to be used as a black box precluding the need to redo the proof of the main absorption theorem. However, Delcourt and Postle's proof of Theorem~\ref{thm:EfficientOmniAbsorber} proceeds via induction on the uniformity and so requires a certain stronger `refinedness' property of the omni-absorber for the purposes of induction. It is this refinedness property that makes the method black-box. Here is the key definition.

\begin{definition}[$C$-refined]
A $K_q^r$-omni-absorber $A$ is \emph{$C$-refined} if every edge of $X\cup A$ is in at most $C$ elements of the decomposition family $\mathcal{F}_A$.    
\end{definition} 

Here then is the true black-box theorem from the work of Delcourt and Postle~\cite{DPI} which generalizes Theorem~\ref{thm:EfficientOmniAbsorber} by proving the existence of \textbf{refined} efficient omni-absorbers. (We note Keevash asserts that the omni-absorbers in his new short proof are also $C$-refined for some constant $C$ but does not explicitly derive this.)

\begin{thm}[Refined Efficient Omni-Absorber~\cite{DPI}]\label{thm:RefinedEfficientOmniAbsorber}
$\forall~q > r \ge 1$, $\exists~C\ge 1$: If $X\subseteq  K_n^r$ and $\Delta_{r-1}(X) \le \frac{n}{C}$, then there exists a $C$-refined $K_q^r$-omni-absorber $A$ for $X$ with $\Delta_{r-1}(A) \le C\cdot \max\left\{\Delta_{r-1}(X),~n^{1-1/r}\cdot \log n\right\}$.     
\end{thm}

The key idea then is that Theorem~\ref{thm:RefinedEfficientOmniAbsorber} can be used a black box -- a template for the absorption in the various problems mentioned above. Indeed, Theorem~\ref{thm:RefinedEfficientOmniAbsorber} can be viewed as a generalization of Montgomery's \emph{template method}~\cite{M19b}, also called \emph{distributive absorption}, to the setting of designs (his key theorem provides a refined efficient omni-absorber theorem in the $1$-uniform case but in the much harder partite setting; meanwhile the base $1$-uniform case of Theorem~\ref{thm:RefinedEfficientOmniAbsorber} is much easier and admits a simple construction). Distributive absorption has been used to great success in various hypergraph matching problems in recent years and so perhaps it is no surprise that Theorem~\ref{thm:RefinedEfficientOmniAbsorber} would similarly have wide applications. 

How does one modify the omni-absorber for the settings above? For high minimum degree, this is accomplished by \emph{re-embedding} the omni-absorber from Theorem~\ref{thm:RefinedEfficientOmniAbsorber} into the host graph; this only requires showing that certain gadgets we call \emph{fake edges} as well as absorbers can be embedded into the host graph and so becomes an embedding problem.

For the other applications of high girth and probability thresholds, the key object for our applications is what we call a \emph{$K_q^r$-booster} which is essentially a non-trivial $K_q^r$-absorber of a clique $K_q^r$ defined as follows.

%Define booster... 
\begin{definition}[Booster]Let $q > r\ge 1$ be integers. A \emph{$K_q^r$-booster} is an $r$-graph $B$ along with two $K_q^r$-decompositions $\mathcal{B}_{\rm on},\mathcal{B}_{\rm off}$ of $B$ such that $\mathcal{B}_{\rm on}\cap \mathcal{B}_{\rm off} =\emptyset$ (that is, they do not have a $K_q^r$ in common). \end{definition}

While this may at first seem extraneous, the idea is that by embedding boosters on the cliques of the decomposition family of our omni-absorber, we may exchange the cliques given by our black-box theorem for cliques with more desirable properties. To that end, it is helpful to think of a booster as \emph{rooted} by specifying a copy $S$ of $K_q^r$ in $\mathcal{B}_{\rm off}$; then if we attach a copy of $B\setminus S$ at $S$, this can be used to boost the properties of the clique $S$. Hence when building a decomposition, if we originally wanted to use the clique $S$ in the decomposition $\mathcal{Q}_A(L)$, we could instead replace it with $\mathcal{B}_{{\rm on}}$ which decomposes $B$; where as if we do not desire to use $S$, we simply use $\mathcal{B}_{\rm off}\setminus \{S\}$ to decompose the edges of $B\setminus S$. In this way, we can replace $S$ with the cliques of the two decompositions $\mathcal{B}_{\rm on}$ and $\mathcal{B}_{\rm off}\setminus \{S\}$ which may have more desirable properties than $S$ (such as being high girth or fairly random). We note this embedding of a separate booster for each clique of the decomposition family \emph{is only possible since the number of cliques of the decomposition family is small as the omni-absorber is $C$-refined!}

Thus for the various applications, the absorption step has been black-boxed with the new challenge being various embedding problems and ensuring the omni-absorber has the desired properties. Intriguingly, this usually means the various solutions developed in each problem can be layered on top of each other --- that our omni-absorber can be made to simultaneously have the desired properties. \emph{This could thus allow the proof of one unifying theorem of designs}, a monumental project we have begun by somewhat uniting the first two problems and the dream of which we discuss further in the concluding remarks.

\section{Proof of the Existence Conjecture: A Refined Absorption Framework.}

Now that we have discussed the ideas, we can present the outline of the refined absorption proof of the Existence Conjecture, writing the steps in the order they should be performed as follows. 

\vskip.1in
\noindent {\bf Proof structure for the `refined absorption' proof of the Existence Conjecture:}
\vskip.1in
\begin{enumerate}\itemsep.1in
    \item[(1)] {\bf `Reserve'} a random subset $X$ of $E(K_n^r)$ uniformly with some small (well-chosen) probability $p$.
    \item[(2)] Construct an {\bf omni-absorber} $A$ of $X$ with small maximum $(r-1)$-degree via Theorem~\ref{thm:RefinedEfficientOmniAbsorber}.
    \item[(3)] {\bf ``Regularity boost''} $J:=K_n^r\setminus (A\cup X)$, namely by finding a set of cliques of $J$ such that every edge in $J$ is in roughly the same (dense) number of cliques in the set. 
    \item[(4)] Apply the {\bf ``Nibble with Reserves''} theorem to find a $K_q^r$ packing of $K_n^r\setminus A$ covering $K_n^r\setminus (A\cup X)$ and then extend this to a $K_q^r$-decomposition of $K_n^r$ by definition of omni-absorber.  
\end{enumerate}

\noindent \emph{How difficult are each of these steps to prove?} Step (1) is quite simple (a few paragraphs), choosing each edge independently with probability $p$ for a well-chosen $p$ and then using concentration inequalities to show that with high probability $X$ has the desired properties (low maximum degree and that every edge of $K_n^r$ is in many cliques with edges of $X$). The following lemma suffices for step (1).

\begin{lemma}[Reserves]\label{lem:RandomX}
$\forall~q > r\ge 1$, $\exists~\gamma \in (0,1)$ s.t. for all $n$ large enough: If $p\in [n^{-\gamma},1)$, then there exists $X\subseteq K_n^r$ with $\Delta(X)\le 2pn$ such that every $e\in K_n^r\setminus X$ is in at least $\frac{1}{2}\cdot p^{\binom{q}{r}-1}\cdot \binom{n}{q-r}$ $K_q^r$-cliques of $X\cup \{e\}$.
\end{lemma}

Step (2), namely the proof of Theorem~\ref{thm:RefinedEfficientOmniAbsorber}, is the longest (and newest) and for us requires our method of refined absorption (Keevash's new proof~\cite{K24} provides an alternate method to proving the required theorem from step (2) while the other steps of Keevash's new proof are essentially the same). First though we also require a proof of the existence of absorbers (Theorem~\ref{thm:AbsorberExistence}) as discussed in the introduction. We discuss this and the proof of Theorem~\ref{thm:RefinedEfficientOmniAbsorber} in the next section. 

Step (3) follows from a special case of the ``Boost Lemma'' (Lemma 6.3 in~\cite{GKLO16}) of Glock, K\"{u}hn, Lo, and Osthus~\cite{GKLO16}; the proof while ingenious is relatively short (a few pages). Namely one can use that to show that a graph $G$ with minimum degree $\delta(G)\ge (1-\varepsilon)n$ for some small enough $\varepsilon$ has a very regular set of cliques - meaning every edge is in roughly the same amount of cliques where the irregularity is polynomially small (as required by ``nibble with reserves''). This essentially follows by showing that high minimum degree hypergraph has a low-weight fractional decomposition (discussed later in the context of Nash-Williams' Conjecture) and then sampling from the resulting distribution. The following lemma then suffices for a proof of the Existence Conjecture.

\begin{lemma}[Boost]\label{lem:RegBoost}
$\forall~q>r\ge 1$, $\exists~\varepsilon \in (0,1)$ s.t.~for all $n$ large enough: If $J\subseteq K_n^r$ with $\delta(J) \ge (1-\varepsilon)n$, then there exists $\mathcal{H} \subseteq K_q^r(J)$ such that every $e\in J$ is in  $\left(\frac{1}{2}\pm n^{-1/3}\right)\cdot \binom{n-r}{q-r}$ cliques of $\mathcal{H}$.
\end{lemma}

Step (4) is a generalization of the `base' nibble theorem. We note this step could also be done ad hoc via a combination of more advanced nibble theorems and the Lov\'asz Local Lemma (see Keevash~\cite{K24}). Nevertheless, it is useful to codify the two steps of nibble and reserves as one package since a) for the various applications, we will require properties (high girth, randomness) of both steps that cannot be ensured when considering them separately, and b) codifying general black-box theorems for the combined step allows a more unified and less ad hoc approach. We discuss this further in the next subsection.

\subsection{Nibble with Reserves for Designs.}

As mentioned, in order to build an absorber, we desire to force the leftover of nibble into a special reserve set $X$. One can prove an even more general ``Nibble with Reserves'' Theorem for hypergraph matchings. Namely, if one desires a hypergraph matching covering all vertices in $G$ it suffices to have access to an additional reserve hypergraph $G'$ whose edges contain exactly one vertex from $G$. Provided that every vertex in $G$ has large degree in $G'$ and $G'$ has not too large a maximum degree, such a matching can be found. Delcourt and Postle formulated a very general hypergraph matching version of this idea (see~\cite{DPI} for a formal statement) which even implies the coloring version of nibble (the Pippenger-Spencer Theorem~\cite{PS89}) and the even more general list coloring version (Kahn's Theorem~\cite{K96}). 

We omit a statement of the general theorem for brevity but here is that very general theorem specified to designs, where the first hypergraph is the standard ``Design Hypergraph'' of $K_n^r$ (whose vertices are the edges of $K_n^r$ and whose edges are the cliques of $K_n^r$) and the extra reserve hypergraph is the ``reserve cliques''.

\begin{theorem}[Nibble with Reserves for Designs]\label{thm:NibbleReservesDesigns}
$\forall~q > r\ge 1,~\beta \in (0,1)$, $\exists~\alpha\in (0,1)$ s.t.~for all $D$ large enough:
Let $G$ and $\mathcal{H}$ be as in Theorem~\ref{thm:NibbleDesigns}. If there exists $X\subseteq K_n^r\setminus G$ and $\mathcal{H}' \subseteq K_q^r(G\cup X)$ where $|H\cap E(G)|=1$ for all $H\in \mathcal{H}'$ and s.t.~every $e\in G$ is in at least $D^{1-\alpha}$ elements of $\mathcal{H}'$ and every $e\in X$ is in at most $\frac{D}{\beta}$ elements of $\mathcal{H'}$, then there exists a $K_q^r$-packing $G\cup X$ that covers all edges of $G$. 
\end{theorem}

Once more, we only need the above theorem when $D$ is some linear proportion. Hence setting $D:=p \binom{n}{q-r}$ for some fixed constant $p\in (0,1)$ (and fixing $\beta$ small enough) yields the following corollary.

\begin{corollary}\label{cor:NibbleReservesSpecific}
$\forall~q > r\ge 1,~\rho \in (0,1]$, $\exists~\alpha\in (0,1)$ s.t.~for all $n$ large enough: Let $G$ and $\mathcal{H}$ be as in Corollary~\ref{cor:NibbleDesignsSpecific}. If there exists $X\subseteq K_n^r\setminus G$ such that every $e\in G$ is in at least $n^{q-r-\alpha}$ $K_q^r$-cliques of $X\cup \{e\}$, then there exists a $K_q^r$-packing of $G\cup X$ that cover all edges of $G$. 
\end{corollary}

As hinted, Corollary~\ref{cor:NibbleReservesSpecific} does not need the full generality of nibble with reserves. Rather it can more directly be deduced by proving that the leftover provided for designs by the standard nibble has small maximum degree (as shown by Alon and Yuster~\cite{AY05}) and then using the Lov\'asz Local Lemma and reserve cliques to extend the packing to cover all edges of $G$. However, this idea can also be generalized to prove the full nibble with reserves. 

\subsection{Proof of Existence.} The proof of Existence for the $\lambda=1$ case (the hardest case) now essentially follows immediately by combining Lemma~\ref{lem:RandomX} (reserves), Theorem~\ref{thm:EfficientOmniAbsorber} (efficient omni-absorbers), Lemma~\ref{lem:RegBoost} (regularity boosting), and Corollary~\ref{cor:NibbleReservesSpecific} (nibble with reserves for designs) as in the proof outline. The $\lambda\ge 1$ case similarly follows with a few additional minor modifications (mostly involving that the host graph is a multigraph).

\section{Proof of the Refined Efficient Omni-Absorber Theorem.}

\subsection{Short Construction of Absorbers.}

In this subsection, we present the construction of $K_q^r$-absorbers from~\cite{DKP24}. However, we present a different interpretation of the construction that might be conceptually somewhat simpler.

\emph{How do we build a $K_q^r$-absorber?} Naturally, it is constructive to consider an easier problem. First let us not concern ourselves with whether $V(L)$ is independent in $A$; second, and more crucially, let us pass to the multi-hypergraph setting. Here we could even add isolated vertices to $L$ so that $V(L)=V(A)$ in the end.
\vskip.1in
\noindent \emph{What then is such a ``multi-absorber"?} We have a multigraph $A$ on vertex set $V(L)$ with a $K_q^r$-decomposition $\mathcal{Q}_1$ of $L\cup A$ and a $K_q^r$-decomposition $\mathcal{Q}_2$ of $A$. It turns out this is equivalent to the old and crucial notion of an \emph{integral $K_q^r$-decomposition} of $L$ defined as follows (one views the positive cliques of the integral decomposition as the decomposition of $L\cup A$ and the negative cliques as the decomposition of $A$). 

\begin{definition}[Integral Decomposition]
An \emph{integral $K_q^r$-decomposition} $\Phi$ of an $r$-graph $L$ is an assignment of integers to the copies of $K_q^r$ in $K_{V(L)}^r$ such that $\sum_{Q\ni e} \Phi(Q) = +1$ for all $e\in L$ and $\sum_{Q\ni e} \Phi(e) = 0$ for all $e\in K_{V(G)}^r\setminus L$.
\end{definition}

Graver and Jurkat~\cite{GJ73} and independently Wilson~\cite{W73} proved the following. 

\begin{theorem}[Graver and Jurkat~\cite{GJ73}, Wilson~\cite{W73}]\label{thm:Integral}
Let $q > r\ge 1$ be integers. If $L$ is a $K_q^r$-divisible $r$-graph with $v(L)\ge q+r$, then there exists an integral $K_q^r$-decomposition of $L$.
\end{theorem}

\noindent \emph{How then do we convert a multi-absorber to an absorber?} We could define $A$ as the edges (including multiples) of the negative cliques (equivalently $A\cup L$ as the positive cliques). If $A$ had no multiple edges and $V(L)$ was independent in $A$, then $A$ would be an absorber as desired. Otherwise, we could replace each edge of $A$ with a `fake edge' (a gadget with the same divisibility properties as an edge), this would remove multi-edges and yield the independence property (if embedded into a larger set of vertices). However, the decompositions would now be into ``fake-ish'' cliques (i.e. cliques where some edges are replaced with fake edges). If we had a construction of absorbers for ``fake-ish'' cliques, we could then embed these to finish.  

A similar but ultimately simpler alternative is to transpose to a `negative' setting. Namely, we instead add an \emph{anti-edge} $e^-$ (i.e.~a clique minus the edge $e$, also denoted $\textrm{AntiEdge}_q^r(e)$) for each edge $e\in L\cup A$ and let $A^* = \bigcup_{e\in L\cup A} e^-$. Then the $K_q^r$-decomposition of $L\cup A$ (the positive one) corresponds to an \emph{anti-clique} (i.e.~anti-edges for all edges of a clique) decomposition of $A^*$; similarly the $K_q^r$-decomposition of $A$ (the negative one) corresponds to an anti-clique decomposition of $\bigcup_{e\in A} e^-$ while $L\cup \bigcup_{e\in L} e^-$ naturally decomposes into $K_q^r$-cliques $(e\cup e^-: e\in L)$. Thus if we could just build a $K_q^r$-absorber for an anti-clique, then by attaching absorbers to all the anti-cliques above would yield the desired $K_q^r$-absorber for $L$. [Note this is simpler than fake edge approach as we to only need to find an absorber for one graph -- the anti-clique.]
\vskip.1in
\noindent \emph{How then do we build an absorber for an anti-clique?} First we construct a \emph{booster} $B$. This is done via a special matrix construction using Cauchy matrices. Fixing a root clique $S$ of $B$, we then layer copies of $B$ to create an \emph{orthogonal} booster $B_o$ (where orthogonal means that every $e\in S$ is in a different clique $Q_e$ of $(B_{o})_{\rm on}$) -- since if attached appropriately each new layer can separate two edges that were in the same clique previously. Finally, we attach copies $B_e$ of $B_o$ to each $Q_e$ with $e\in S$. Then the cliques $Q_e'$ of $(B_e)_{\rm on}$ containing $e$ are internally vertex-disjoint anti-edges so that $\bigcup_{e\in S} Q_e'$ forms an anti-clique; lastly, one checks that said anti-clique is independent among the remaining edges and they yield the desired absorber decompositions.

\begin{remark}
On one final note, for using the general embedding lemma presented later, it is useful if a $K_q^r$-absorber has the additional property of being \emph{edge-intersecting} - that is, for every edge $e\in A$, there exists $f\in L$ such that $e\cap V(L)\subseteq f$. We note that the existence of edge-intersecting absorbers follows from the construction above provided the initial edge-decomposition (i.e.~multi-absorber) is also edge-intersecting [since boosters are edge-intersecting]. To prove the existence of an edge-intersecting multi-absorber, one first pushes $L$ one uniformity at a time onto a new set of vertices (in an edge-intersecting way) and only then finds a multi-absorber. We omit the details. Indeed, this additional property greatly simplifies what must be tracked in the refine down procedure - removing the tracking of ``refinement degree'' in~\cite{DPI} and thus simplifying the whole proof a good deal. 
\end{remark}

\subsection{Refiners and Multigraphs.} Now we move on towards presenting the ideas in the proof of Theorem~\ref{thm:RefinedEfficientOmniAbsorber}. \emph{How do we build a $C$-refined efficient omni-absorber?} The first key realization (as in the absorber construction above) is that we may reduce the problem of finding a $C$-refined efficient omni-absorber to its multigraph version wherein we allow parallel edges. Note that we still must ensure said ``multi-omni-absorber'' is $C$-refined and efficient in its maximum degree (where parallel edges count with multiplicity towards degree). The reduced problem then is not much simpler but at least we will not have to worry about embedding substructures edge-disjointly since we allow multiple edges.
\vskip.1in

\noindent \emph{How to reduce to a multi-omni-absorber?} For our proof of the existence of $K_q^r$-absorbers, we embedded anti-edges and then used anti-clique absorbers. Now that we have access to absorbers for all $K_q^r$-divisible $L$, we no longer have to transfer to the `negative realm' but rather can replace multiple edges with `fake-edges' (an anti-anti-edge). Namely a fake-edge,which could be written succinctly as $\bigcup_{f\in e^-} f^{-}$, is formally defined as follows. 

\begin{definition}[Fake-Edge]\label{def:FakeEdge}
Let $q>r\ge 1$. Let $f$ be a set of vertices of size $r$. A \emph{fake-edge on $f$}, denoted ${\rm FakeEdge}_q^r(f)$, is a set of new vertices $x_1,\ldots, x_{q-r}$ together with a set of anti-edges $\{ {\rm AntiEdge}_q^r(T): T\in \binom{f\cup \{x_i:i\in[q-r]\} }{r} \setminus \{f\} \}$.
\end{definition}

We note that a fake-edge has the same divisibility properties as an actual edge. Thus we will replace each multi-edge in our multi-omni-absorber by embedding a fake-edge in its stead; then we will embed a `private' (read: edge-disjoint) absorber for each ``fake-ish'' clique in the ensuing decomposition family. Now we have to ensure that the fake-edges (and similarly that private absorbers) are edge-disjoint and that combined they do not have too large a maximum degree. This will be accomplished by means of an embedding lemma; as we will also require embedding lemmas in our applications, we present a very general embedding lemma, Lemma~\ref{lem:EmbedMinDegree}, in the next subsection. This embedding lemma is of its own nature with a different flavor than the arguments to find a refined efficient multi-omni-absorber. This combination of ideas allows us to quarantine all embedding issues to one final embedding step (instead of having to consider embedding issues at each step of Refine Down). Since we now will embed `private' (i.e.~edge-disjoint) absorbers for all the fake cliques in our decomposition, we now ask the following potent question.

\vskip.1in

\noindent \emph{Why not reduce to decomposing into small divisible graphs instead of decomposing into cliques?} Now that we have proved the existence of $K_q^r$-absorbers for all $L$, we can embed such private absorbers for all $L$ not just anti-cliques or fake-cliques. This motivates the following definition.

\begin{definition}[Refiner]
Let $q > r\ge 1$ be integers. Let $X$ be an $r$-graph. We say an $r$-graph $R$ is a \emph{$K_q^r$-refiner} of $X$, if $V(X)=V(R)$, $X$ and $R$ are edge-disjoint, and for every $K_q^r$-divisible $L\subseteq X$, there exists a decomposition of $R\cup L$ into $K_q^r$-divisible subgraphs. 
\end{definition}

Of course we should take care that such $L$ are small, so that there will be room to embed their absorbers. [Indeed the existence of refiners is trivial if we allow unbounded subgraphs since then the empty graph is a refiner!] Thus we amend the definition of $C$-refined to also include this assumption. Now it suffices via embedding fake-edges and absorbers (via the general embedding lemma) to prove Theorem~\ref{thm:RefinedEfficientOmniAbsorber} only for a $C$-refined refiner and even allowing multi-edges. There is still much work to be done but this is now the state.

\subsection{An Embedding Lemma for Supergraph Systems.}

In this subsection, we state a general embedding lemma which can be used to embed fake edges and (private) absorbers as required above. Indeed, both steps follow from a general embedding lemma about embedding rooted graphs $\mathcal{W}=(W_H:H\in \mathcal{H})$ onto a $C$-refined family $\mathcal{H}$ of subgraphs of an $r$-graph $R$ of small maximum degree in a manner that is edge-disjoint and has low maximum degree. For `fake edges', the family $\mathcal{H}$ will simply consist of each multi-edge $e$ to be replaced and we will set $W_e := {\rm FakeEdge}_q^r(e)$. For `private absorbers', the family $\mathcal{H}$ will be the decomposition family of $R$ where $R$ is a $K_q^r$-refiner and $W_H$ will be a $K_q^r$-absorber of $H$.

Realizing that these two embedding problems have the same nature and follow from a general embedding lemma would already be enough motivation to codify such a lemma. But in fact, the general embedding lemma and its variants feature prominently in the method's applications. Indeed for applications to problems in extremal and probabilistic design theory, including high minimum degree, high girth, and thresholds to be discussed later in this article, the crucial difficulty of these problems given our black-box omni-absorber theorem will be the associated embedding problem. Namely, how do we embed absorbers into high minimum degree host graphs, or how do we embed boosters for our omni-absorber in a high-girth or spread way? 

Given all these applications, we believe it to be imperative to codify this embedding problem in a general fashion. Here we only mention a very general form of the embedding lemma, though one of its known proof via the Lov\'asz Local Lemma can be extended to even more general forms (the most general statement would essentially be ``whatever the Local Lemma gives''). That said, we hope the required conditions of this version can be understood well enough for further applications.  
Let us then start to define this general embedding problem. In all our applications, we will have a (sparse) hypergraph $R$ to which we want to embed graphs in a rooted way to subgraphs of $R$, often with the additional restriction of embedding said graphs inside some host graph $G$. This motivates the following definitions. First recall that a hypergraph $W$ is a \emph{supergraph} of a hypergraph $J$ if $V(J)\subseteq V(W)$ and $E(J)\subseteq E(W)$.

\begin{definition}[Supergraph Systems]
Let $\mathcal{H}$ be a family of subgraphs of a hypergraph $J$. A \emph{supergraph system} $\mathcal{W}$ for $\mathcal{H}$ is a family $(W_H : H\in \mathcal{H})$ where for each $H\in \mathcal{H}$, $W_H$ is a supergraph of $H$ with $(V(W_H)\setminus V(H))\cap V(J)=\emptyset$ and for all $H'\ne H\in \mathcal{H}$, we have $V(W_H)\cap V(W_{H'}) \setminus V(J)= \emptyset$. We let $\bigcup \mathcal{W}$ denote $\bigcup_{H\in \mathcal{H}} W_H$ for brevity. We say $\mathcal{W}$ is \emph{edge-intersecting} if $W_H$ is $J$-edge-intersecting for every $H\in \mathcal{H}$. For a real $C \geq 1$, we say that $\mathcal{W}$ is \emph{$C$-bounded} if $\max\{e(W_H),~v(W_H)\}\le C$ for all $H\in \mathcal{H}$. 
\end{definition}

All of the problems above can then be thought of as embedding supergraph systems appropriately. For our embedding lemma, it proves useful if the supergraphs are edge-intersecting (indeed, this is why we proved the existence of edge-intersecting absorbers). We now formalize the notion of embedding a supergraph system into a host graph $G$ as follows.

\begin{definition}[Embedding a Supergraph System]
Let $J$ be a hypergraph and let $\mathcal{H}$ be a family of subgraphs of $J$. Let $\mathcal{W}$ be a supergraph system for $\mathcal{H}$. Let $G$ be a supergraph of $J$. An \emph{embedding} of $\mathcal{W}$ \emph{into} $G$ is a map $\phi : V(\bigcup \mathcal{W}) \hookrightarrow V(G)$ preserving edges such that $\phi(v)=v$ for all $v\in V(J)$. We let $\phi(\mathcal{W})$ denote $\bigcup_{e\in \bigcup W} \phi(e)$ (i.e.~the subgraph of $G$ corresponding to $\bigcup \mathcal{W}$).
\end{definition}

We are now prepared to state the general embedding lemma (see~\cite{DHLP25} for a proof). 

\begin{lemma}[General Embedding]\label{lem:EmbedGeneral}
$\forall~C > r \ge 1,~\gamma\in (0,1]~\exists~C'\ge 1$: Let $G\subseteq K_n^r$ and let $J$ be a multi-$r$-graph with ${\rm Simple}(J)\subseteq G$ and $\Delta(J)\le \frac{n}{C'}$. Suppose that $\mathcal{H}$ is a $C$-refined family of subgraphs of $J$ and $\mathcal{W}$ is a $C$-bounded edge-intersecting supergraph system for $\mathcal{H}$. If for each $H\in \mathcal{H}$ there exist at least $\gamma\cdot n^{|V(W_H)\setminus V(H)|}$ embeddings of $W_H$ into $G\setminus (E(J)\setminus E(H))$,
then there exists an embedding $\phi$ of $\mathcal{W}$ into $G$ such that $\Delta_{r-1}(\phi(\mathcal{W})) \le C'\cdot \Delta_{r-1}(J)$.
\end{lemma}

The conditions in the lemma above are quite natural. Of course, we need $J$ to have somewhat small maximum degree if we hope to have space to embed the supergraph system. For similar reasons, we want the supergraphs to be $C$-bounded in order to fit them. The $C$-refined and edge-intersecting properties meanwhile also serve to not overburden our embedding choices. The $C$-refined property is useful since otherwise too many supergraphs might need to be embedded coming off the same edge. Similarly, without the edge-intersecting property we would have to worry for $r\ge 3$ about an $(r-1)$-set being in too many supergraphs of the system (and if each requires an edge containing said set this is indeed an issue).

We note that Lemma~\ref{lem:EmbedGeneral} has the following corollary when $G$ is a high minimum degree hypergraph.

\begin{lemma}[Embedding]\label{lem:EmbedMinDegree}
$\forall~C > r \ge 1,~\exists~C'\ge 1,~\varepsilon\in (0,1)$: Let $G\subseteq K_n^r$ with $\delta(G)\ge (1-\varepsilon)n$. Let $J$ be a multi-$r$-graph with ${\rm Simple}(J)\subseteq G$ and $\Delta(J)\le \frac{n}{C'}$. If $\mathcal{H}$ is a $C$-refined family of subgraphs of $J$ and $\mathcal{W}$ is a $C$-bounded edge-intersecting supergraph system for $\mathcal{H}$, then there exists an embedding $\phi$ of $\mathcal{W}$ into $G$ such that $\Delta_{r-1}(\phi(\mathcal{W})) \le C'\cdot \Delta_{r-1}(J)$.
\end{lemma}
\begin{proof}
This follows from Lemma~\ref{lem:EmbedGeneral} by choosing $\varepsilon$ small enough so that $|\Phi(H)|\ge \gamma \cdot n^{|V(W_H)\setminus V(H)|}$ for all $H\in \mathcal{H}$ for some $\gamma\in (0,1)$.
\end{proof}
 
\subsection{Refine Down.}

To prove Theorem~\ref{thm:RefinedEfficientOmniAbsorber}, one first proves the $1$-uniform case via the following simple yet clever construction: we let $A$ be a set of $q\cdot |X|$ new vertices; we let $H_1$ be a \emph{tight $q$-path} (i.e.~every $r$ consecutive vertices form an edge) on $A$ and we let $H_2$ be a tight $q$-path on $A\cup X$ (with each vertex of $X$ interspersed every $q$ vertices of $A$); then we let $\mathcal{F}_A:= H_1\cup H_2$. Meanwhile for $r\ge 2$, the plan is to reduce down to a constant set of vertices (at which point the theorem trivially holds). To that end, one defines a refiner with \emph{remainder} (which is allowed to contain a `leftover' of undecomposed edges). The key is to prove the Refine Down Lemma which constructs an efficient refined multi-refiner whose remainder lies inside a smaller set of vertices. The Refine Down Lemma is similar in spirit to the Cover Down Lemma and also proceeds by induction on uniformity. That said, instead of needing to cover all edges, we just have to absorb those in the remainder (and are allowed to use multi-edges and to decompose into small divisible graphs instead). Nevertheless, there is a cost to the efficiency of this lemma based on how small this set is (the smaller the set the higher the max degree due to compression factors). Given the cost, we will not be able to apply the Refine Down Lemma to reduce down to a constant sized set without an additional key idea.
\vskip.1in
\noindent \emph{How to overcome this difficulty?} The final ingredient is to apply the $1$-uniform construction to sets of parallel edges so as to leave a remainder with constant multiplicity; this is done in the Multiplicity Reduction Lemma. Then the plan is to use the Refine Down Lemma to refine down to the case where the remainder is essentially complete (i.e. max degree is roughly the number of vertices). At that point, we will alternate steps between the Refine Down Lemma to a vertex set that is some constant factor smaller (which then increases the multiplicity by some large constant factor) and the Multiplicity Reduction Lemma (which then decreases the multiplicity to some fixed constant and hence the max degree as well). We are finished when the set reaches constant size.

\section{Extremal Design Theory: Nash-Williams' Conjecture.}

\subsection{Triangle Decompositions in High Minimum Degree Graphs.} A natural question is when a partial Steiner system (i.e.,~a set of edge-disjoint cliques) can be extended to a Steiner system. More generally, one may wonder whether `nearly' complete graphs admit a $K_q^r$-decomposition provided they are $K_q^r$-divisible. The \emph{decomposition threshold of $K_q^r$}, denoted by $\delta_{K_q^r}$, is defined as $\lim\sup_{n\rightarrow \infty} \delta_{K_q^r}(n)/n$ where $\delta_{K_q^r}(n)$ is the smallest integer $d$ such that every $K_q^r$-divisible graph $G$ on $n$ vertices with $\delta(G)\ge d$ has a $K_q^r$-decomposition. Note one must impose some condition, such as high minimum degree, to ensure that every edge is in at least one clique. The most central and famous conjecture in this direction is the following conjecture of Nash-Williams from 1970~\cite{NW70}.

\begin{conj}[Nash-Williams~\cite{NW70}]\label{conj:NW}
The following holds for sufficiently large $n$: If $G$ is a $K_3$-divisible graph on $n$ vertices with $\delta(G)\ge \frac{3}{4}n$,  then $G$ admits a $K_3$-decomposition.
\end{conj}

In an addendum to Nash-Williams' article~\cite{NW70}, Graham provided a construction showing that the fraction $3/4$ is necessary. Indeed, his construction also provides graphs with minimum degree $\frac{3}{4}n-1$ and no fractional $K_3$-decomposition. Formally, a \emph{fractional $K_q^r$-decomposition} of an $r$-graph $G$ is an assignment of non-negative weights to each copy of $K_q^r$ in $G$ such that the sum of the weights along each edge is exactly $1$. The \emph{fractional $K_q^r$-decomposition threshold} $\delta^*_{K_q^r}$ is defined as the infimum of all real numbers $c$ such that every $r$-graph $G$ with minimum degree at least $c\cdot v(G)$ admits a fractional $K_q^r$-decomposition. Since a $K_3$-decomposition is also a fractional $K_3$-decomposition, the existence of a fractional $K_3$-decomposition is a necessary condition for the existence of a $K_3$-decomposition. Graham's construction thus shows that the fraction $3/4$ would also be tight for the \emph{fractional relaxation} of Nash-Williams' Conjecture (namely $\delta_{K_3}^*\ge 3/4$). 

Nash-Williams' Conjecture may be seen as a foundational problem in what may be termed \emph{extremal design theory}; it is instructive to compare this value to the fractions in the minimum degree thresholds for other extremal problems:~namely the threshold is $1/2$ for the existence of at least one triangle (which follows from a classical theorem usually attributed to Mantel~\cite{mantel} from 1907) and $2/3$ for the existence of a \emph{triangle factor}, that is a set of vertex-disjoint triangles spanning all vertices (as given by the celebrated Corr\'adi--Hajnal Theorem~\cite{CH63}). 

Indeed, one could view these results of Corr\'adi--Hajnal~\cite{CH63}, Hajnal--Szemer\'edi~\cite{HS70}, and their ilk, but also results such as the famous result of R\"odl--Ruci\'nski--Szemer\'edi~\cite{RRS06} on the minimum degree threshold for hypergraph matchings, as the fundamental questions of $1$-uniform extremal design theory (covering all the vertices); from that perspective, one could also view all of Tur\'an theory as $0$-uniform extremal design theory (since covering the empty set with cliques is equivalent to finding one clique). In this way, {\bf Nash-Williams' Conjecture is a central question in extremal design theory} as borne out by its generalizations (to $K_q$ or general uniformities) and reductions to it.

The current best-known results on Nash-Williams' Conjecture are as follows. In 2016, Barber, K\"{u}hn, Lo, and Osthus~\cite{BKLO16} proved the following.

\begin{thm}[Barber, K\"{u}hn, Lo, and Osthus~\cite{BKLO16}]\label{thm:BKLO}
For each real $\varepsilon>0$, every sufficiently large $K_3$-divisible graph~$G$ on~$n$ vertices with 
    $\delta(G)\ge (\max\{\delta^*_{K_3},\frac{3}{4}\}+\varepsilon)n$
    admits a $K_3$-decomposition. 
\end{thm}

Meanwhile, the best-known fractional result of Delcourt and Postle~\cite{DP2021progress} from 2021 is as follows. 

\begin{thm}[Delcourt and Postle~\cite{DP2021progress}]\label{thm:DP2021}
If $G$ is a $K_3$-divisible graph on~$n$ vertices with $\delta(G)\ge (\frac{7+\sqrt{21}}{14})n$, then $G$ 
    admits a fractional $K_3$-decomposition. 
\end{thm}

Combined with the previous theorem this confirms that Nash-Williams' Conjecture holds for graphs with minimum degree at least $0.82733n$. We note that Theorem~\ref{thm:DP2021} improved a series of earlier works on the fractional relaxation by Garaschuk~\cite{garaschuk2014linear} in 2014, Dross~\cite{dross2016fractional} in 2016, and Dukes and Horsley~\cite{dukes_minimum_2020} in 2020. Nash-Williams' Conjecture has become a central open question in design theory as various generalizations are now tied to it (e.g., the minimum degree threshold for an $F$-decomposition for any $3$-chromatic $F$~\cite{GKLMO19}, the minimum degree threshold for approximately packing families of $2$-regular graphs~\cite{CKKO19}, and the minimum degree threshold for high-girth triangle decompositions~\cite{DHLP25}). For more history on Nash-Williams' Conjecture and its generalizations, see the survey by Glock, K\"uhn, and Osthus~\cite{GKO20Survey}.

\subsection{Generalizations of Nash-Williams' Conjecture.} A core area of design theory is to consider analogous questions for decompositions into cliques larger than triangles. The following is a folklore generalization of Nash-Williams' Conjecture for $K_q$-decompositions (dating at least to Gustavsson's 1991 thesis~\cite{gustavsson1991decompositions}), a conjectured high minimum degree version of Wilson's Theorem.

\begin{conj}[Folklore]\label{conj:Folklore}
For each integer $q\ge 4$, the following holds for sufficiently large $n$: If $G$ is a $K_q$-divisible graph on $n$ vertices with minimum degree $\delta (G) \geq \left(1-\frac{1}{q+1}\right)n$,  then $G$ admits a $K_q$-decomposition.
\end{conj}

The above conjecture coincides with Nash-Williams' Conjecture when $q=3$. In 1991, Gustavsson~\cite{gustavsson1991decompositions} in his thesis provided a construction to show that this minimum degree value would be tight (this construction was reiterated by Yuster~\cite{Yuster2005asymptotically} in 2005). Indeed, his construction admits no fractional $K_q$-decomposition and so this value would also be tight for the fractional relaxation of the above conjecture. Again, it is instructive to compare this value to the fractions in the minimum degree thresholds for other extremal problems: namely a value of $1-\frac{1}{q-1}$ for the existence of at least one copy of $K_q$ which follows from the classical Tur\'an's Theorem~\cite{turan1941extremal} from 1941, and $1-\frac{1}{q}$ for the existence of a \emph{$K_q$-factor}, that is a set of vertex-disjoint $K_q$'s spanning all vertices, which follows from the famous Hajnal--Szemer\'edi Theorem~\cite{HS70} from 1970.  

Beyond the pleasing symmetry of Conjecture~\ref{conj:Folklore} with the corresponding extremal results noted above, there was even more evidence in favor of the conjecture in the form of the best-known results. %make theorem
In 2019, Glock, K\"{u}hn, Lo, Montgomery, and Osthus~\cite{GKLMO19} proved that if the fractional relaxation of Conjecture~\ref{conj:Folklore} holds for minimum degree $cn$ for some constant $c\ge 1-\frac{1}{q+1}$, then Conjecture~\ref{conj:Folklore} holds for any constant $c' > c$.  Meanwhile, Montgomery~\cite{montgomery_fractional_2019} in 2019 produced the best-known result on the fractional relaxation by obtaining a minimum degree bound of $(1-\frac{1}{100q})n$, notably the first bound whose denominator is linear in $q$. This improved upon a series of earlier works on the fractional relaxation by Yuster~\cite{Yuster2005asymptotically} in 2005, Dukes~\cite{D12} in 2012, and Barber, K\"{u}hn, Lo, Montgomery, and Osthus~\cite{BKLMO17} in 2017. Conjecture~\ref{conj:Folklore} has also become paramount in the realm of extremal design theory as more variations are tied to it (such as the minimum degree threshold for an $F$-decomposition for any $q$-chromatic $F$~\cite{GKLMO19} and the minimum degree threshold for approximately packing certain families of $(q-1)$-regular graphs~\cite{CKKO19}).  

In light of this history, the recent result of Delcourt, Henderson, Lesgourgues, and Postle~\cite{DHLP25b} may then come as a bit of surprise: Conjecture~\ref{conj:Folklore} and even its fractional relaxation are false for all $q\ge 4$ as follows.

\begin{thm}[Delcourt, Henderson, Lesgourgues, and Postle~\cite{DHLP25b}]\label{thm:mainExistence}
$\forall~q\geq 4$, $\exists~c>1$ and infinitely many $K_q$-divisible graphs $G$ with $\delta(G)\ge \big(1-\frac{1}{c\cdot (q+1)}\big)v(G)$ and no fractional $K_q$-decomposition.
\end{thm}

Indeed, the folklore conjecture is not only false but the multiplicative factor is wrong as follows (namely any factor $c$ smaller than $\frac{1+\sqrt{2}}{2}\approx 1.207$ still yields a counterexample provided $q$ is large enough).

\begin{thm}[Delcourt, Henderson, Lesgourgues, and Postle~\cite{DHLP25b}]\label{thm:mainAsymptotic}
    Let $c=\frac{1+\sqrt{2}}{2}$ and fix $\varepsilon>0$. For every large enough integer $q$, there exist infinitely many $K_q$-divisible graphs $G$ with $\delta(G)\ge \big(1-\frac{1}{(c-\varepsilon)\cdot (q+1)}\big)v(G)$ and no fractional $K_q$-decomposition.
\end{thm}

We note that in light of the results above, determining the correct minimum degree threshold for $K_q$-decompositions seems quite a fascinating problem; at the current time, we would not hazard to guess whether the value in Theorem~\ref{thm:mainAsymptotic} is optimal or whether even better constructions may yet be found.

Much less is known regarding hypergraph decompositions. Glock, K\"{u}hn, and Osthus~\cite{GKO20Survey} in their survey on graph decompositions conjectured that $\delta_{K^r_q}=1-\Theta_r(1/q^{r-1})$, where the constant hidden in the $\Theta_r$ notation is allowed to depend on $r$. The best known bound toward this dates from the second proof of Existence in 2016 where Glock, K\"{u}hn, Lo, and Osthus~\cite{GKLO16} proved that $\delta_{K^r_q}\le 1-\Theta_r(1/q^{2r})$. As for the fractional relaxation, the current best known bound is from 2017 by Barber, K\"{u}hn, Lo, Montgomery, and Osthus~\cite{BKLMO17} who proved that $\delta^*_{K_q^r}\leq1-\Theta_r(1/q^{2r-1})$. It would be of interest to improve these bounds. 

\subsection{A Refined Absorption Proof for Nash-Williams' Conjecture.}

We now discuss how to use refined absorption to prove Theorem~\ref{thm:BKLO}. The adaptations actually work to prove the more general version where $G$ is a $K_q$-divisible graph with $\delta(G)\ge \left(\max\{\delta_{K_q}^*, 1-\frac{1}{2q-2}\}+\varepsilon\right)n$ and we desire a $K_q$-decomposition of $G$. One modifies the refined absorption proof framework as follows. For step (1), we `reserve' a random subset $X$ of $E(G)$. Namely, we modify the reserves lemma to hold \emph{inside} $G$ which again follows easily via Chernoff bounds. Notice that lemma below actually works for the better bound of $1-\frac{1}{q+1}+\varepsilon$ in the minimum degree.

\begin{lem}[Reserves for Nash-Williams]\label{lem:RandomXHighMinDeg}
$\forall q \ge 3$ and $\varepsilon\in (0,1)$, $\exists~\gamma \in (0,1)$ s.t.~for all $n$ large enough: Let $G \subseteq K_n$ with $\delta(G)\ge (1-\frac{1}{q+1}+\varepsilon)n$. If $p\in [n^{-\gamma},1)$, then there exists $X\subseteq G$ with $\Delta(X)\le 2pn$ such that every $e\in K_n\setminus X$ is in at least $\frac{1}{(q+1)^q}\cdot p^{\binom{q}{2}-1}\cdot \binom{n}{q-2}$ $K_q$-cliques of $X\cup \{e\}$.
\end{lem}

For step (2), we need to find an omni-absorber $A$ of $X$ \emph{inside $G\setminus X$}. To that end, we find an omni-absorber $A_0$ of $X$ in $K_n$ via Theorem~\ref{thm:RefinedEfficientOmniAbsorber}. Then we \emph{re-embed} $A$ \emph{into $G\setminus X$} using the methodology of fake-edges and private absorbers developed in~\cite{DPI}. Thus the following theorem was proved in~\cite{DHLP25}. 

\begin{thm}[Nash-Williams Refined Efficient Omni-Absorber]\label{thm:NWRefinedEfficientOmniAbsorber}
$\forall~q\ge 3$ and $\varepsilon\in (0,1)$, $\exists~C\ge 1$:
If $X\subseteq  G\subseteq K_n$ with $\delta(G)\ge (1-\frac{1}{2q-2} + \varepsilon)n$ and $\Delta(X) \le \frac{n}{C}$, then there exists a $C$-refined $K_q$-omni-absorber $A$ for $X$ in $G$ with $\Delta(A) \le C\cdot \max\left\{\Delta(X),~\sqrt{n}\cdot \log n\right\}$.    
\end{thm}

The key to proving the above theorem is first that ``fake-edges '' have \emph{rooted degeneracy} at most $q-1$ (meaning an ordering of non-root vertices such that each has at most $q-1$ earlier neighbors in the ordering); the second key is that there exist $K_q$-absorbers of rooted degeneracy at most $2q-2$ (as shown in~\cite{DKPIII} though this also follows from the short construction of absorbers by using boosters of rooted degeneracy $2q-2$ whose construction is explained later in the section of high girth Steiner systems). Note that using such an absorber construction is the bottleneck here (the rooted degeneracy bound of $2q-2$ is tight for $K_q$-absorbers as discussed in~\cite{DKPIII}) but it may be possible to utilize the more sophisticated absorber constructions of~\cite{GKLMO19} to improve the bound above. 

For step (3), a different approach than was originally used in the refined absorption proof of the Existence Conjecture is needed. Indeed, we cannot directly apply the Boosting Lemma from Glock, K\"uhn, Lo, and Osthus~\cite{GKLO16} since that only works if the minimum degree is much closer to $n$, namely there is some $\varepsilon$ where it works for minimum degree $(1-\varepsilon)n$ (but unfortunately this $\varepsilon$ is much smaller than $1/4$). 

The cleanest approach (and by far the easiest) is to use the ``Inheritance Lemma for Minimum Degree'' -- a concept that has echoes in many earlier works in the literature but has started to gain prominence in works on Hamilton cycles, tilings, and now decompositions. We will use a version of the lemma by Lang~\cite{lang2023tiling} as follows. We note that Lang's lemma itself has a remarkably simple proof from standard concentration inequalities that bypasses any use of Szemer\'edi's Regularity Lemma.

\begin{lem}[Inheritance Lemma for Minimum Degree~\cite{lang2023tiling}]\label{lem:LangInheritance}
$\forall~\varepsilon\in(0,1)$ and $m\geq 1$, $\exists~s_0$ s.t.~for all $s\geq s_0$ and $n$ large enough: Let $\delta\geq 0$ and let $G$ be an $n$-vertex graph with $\delta(G)\geq (\delta+\varepsilon)(n-1)$. Then for every m-set $M\subseteq V(G)$, there exist at least $(1-e^{-\sqrt{s}})\binom{n-m}{s-m}$ distinct $s$-sets $S\subseteq V(G)$ s.t.~$M\subseteq S$ and $\delta(G[S])\geq (\delta+\varepsilon/2)(s-1)$.
\end{lem}

From the Inheritance Lemma, we can prove a regularity boosting lemma that works in the setting of Nash-Williams' Conjecture as follows. 

\begin{lem}[Nash-Williams Boosting Lemma -- Constant Irregularity Version]\label{lem:NWRegBoostConstant}
$\forall q\geq 3$ and $\varepsilon \in (0,1)$, $\exists~c\in(0,1)$ s.t.~for all $\gamma\in (0,1)$ and $n$ large enough: If $J\subseteq K_n$ with $\delta(J)\ge (\max\{\delta_{K_q}^*,1-\frac{1}{q+1}\}+\varepsilon)n$, then there exists a family $\mathcal{H}$ of copies of $K_q$ in $J$ such that every $e\in J$ is in $(c\pm \gamma)\cdot \binom{n-2}{q-2}$ copies of $K_q$ in $\mathcal{H}$.
\end{lem}

The idea is that any $s$-subset $S$ of our graph $J$ that satisfies $\delta(J[S]) \ge (\delta_{K_q}^*+\frac{\varepsilon}{2})m$ will admit a fractional $K_q$-decomposition by definition of the fractional decomposition threshold if $s$ is large enough. Thus if we average over these decompositions for all such $S$ and appropriately scale, we will find an ``almost'' fractional $K_q$-decomposition of $J$ (almost meaning each edge has roughly weight 1) but where the weight of any copy of $K_q$ is at most $\frac{C}{n^{q-2}}$ for some constant $C$ only depending on $q$ and $\varepsilon$ (since any fixed copy of $K_q$ is not in too many $S$). Then Lemma~\ref{lem:NWRegBoostConstant} follows by randomly sampling from this fractional weighting and applying Chernoff bounds. Note importantly here that this only gives irregularity (the $\gamma$) for any small but fixed constant $\gamma$ as opposed to a version where $\gamma$ can be made polynomially small in $n$ which we also obtained. Indeed polynomial small irregularity follows by sampling if we can remove the `almost' above. Thus we proved the following stronger theorem (by adding `a fixing' idea of Montgomery's~\cite{montgomery_fractional_2019}) but first a definition. We say a fractional $F$-decomposition of a graph $G$ on $n$ vertices is \emph{$C$-low-weight} if every copy of $F$ in $G$ has weight at most $\frac{C}{n^{v(F)-2}}$.

\begin{thm}\label{thm:LowWeightDecomposition}
    For every graph $F$ and real $\varepsilon\in (0,1)$, $\exists~C\geq 1$ s.t.~for all $n$ large enough: If $G$ is a $n$-vertex graph with $\delta(G)\geq (\delta^*_F+\varepsilon)n$, then there exists a $C$-low-weight fractional $F$-decomposition of $G$.
\end{thm}

\section{Extremal Design Theory, Part II: High Girth Steiner Systems.}

\subsection{Erd\H{o}s' Conjecture on High Girth Steiner Triple Systems.}

Another important stream of research in extremal design theory concerns finding designs that avoid certain substructures, in particular to find Steiner systems that are `locally sparse'. To that end, we say a {\em $(j,i)$-configuration} in a set $\mathcal{P}$ of sets on a ground set $X$ is a set of $i$ elements of $\mathcal{P}$ spanning at most $j$ elements of $X$. Thus a $(j,i)$-configuration in a $K_3$-packing is a set of $i$ triangles spanning at most $j$ vertices. Observe that one triangle is a $(4, 1)$-configuration and any two triangles sharing a vertex is a $(5,2)$-configuration. Indeed, every $(n,3,2)$-Steiner system contains an $(i+3,i)$-configuration for every $1 \le i \le n-3$; this observation prompted Erd\H{o}s to study $(i+2,i)$-configurations in the 1970s. The \textit{girth} of a triangle packing is the smallest integer $i\ge 2$ such that the packing contains an $(i+2, i)$-configuration. Note that a family of triangles having the property of containing no $(4,2)$-configuration is equivalent to being a triangle packing. 

In 1973, Brown, Erd\H{o}s, and S\'{o}s~\cite{BES73} proved that for every $i\geq 2$, there exist constants $C_i>c_i> 0$ such
that every triangle-packing on $n$ vertices of size at least $C_i\cdot n^2$ has an $(i+2,i)$-configuration, while for every integer~$n$ there exists a triangle-packing of size $c_i\cdot n^2$ with no $(i+2,i)$-configuration. We refer the reader to the recent work of Delcourt and Postle~\cite{DP22BES} from 2022 for further information on this problem. Stemming from this work, Erd\H{o}s~\cite{E73} in 1973 made the following conjecture.

\begin{conj}[Erd\H{o}s~\cite{E73}]\label{conj:Erdos}
For every integer $g\geq 3$, every sufficiently large $K_3$-divisible complete graph admits a $K_3$-decomposition with girth at least $g$. 
\end{conj}

In 1993, Lefmann, Phelps, and R{\"o}dl \cite{LPR93} showed that for every $g \ge 2$, there exists a constant $c_g$ depending on $g$ such that for every $n \ge 3$, there exists a partial Steiner triple system $S$ with $|S| \ge c_g \cdot n^2$ and girth at least $g$ (with $c_g \rightarrow 0$ as $g\rightarrow \infty$). More recently, in 2019 Glock, K{\"u}hn, Lo, and Osthus~\cite{GKLO20} and independently Bohman and Warnke~\cite{BW19} settled the approximate version of Erd\H{o}s' Conjecture. In 2022, Kwan, Sah, Sawhney, and Simkin \cite{KSSS2024STS} impressively proved the conjecture in full using the method of iterative absorption.

\begin{thm}[Kwan, Sah, Sawhney, and Simkin~\cite{KSSS2024STS}]\label{thm:KSSS}
Conjecture~\ref{conj:Erdos} is true.
\end{thm}

\subsection{The High Girth Existence Conjecture.} It is natural to consider generalizing the notion of girth to $(n,q,r)$-Steiner systems.  In \cite{GKLO20}, Glock, K{\"u}hn, Lo, and Osthus noted that for every fixed $i \ge 2$, every $(n, q, r)$-Steiner system contains a $((q -r)i+r+1, i)$-configuration. Motivated by this fact, one defines the \emph{girth} of an $(n, q, r)$-Steiner system as the smallest integer $g \ge 2$ for which it contains a $((q - r)g + r, g)$-configuration.  Glock, K{\"u}hn, Lo, and Osthus~\cite{GKLO20}, as well as Keevash and Long~\cite{KL20}, conjectured the existence of $K_q^r$-decompositions of $K_n^r$ of arbitrarily large girth (for $r = 2$ and $q \ge 3$, this was previously asked by F{\"u}redi and Ruszink{\'o} \cite{FR13} in 2013) as follows.

\begin{conj}[High Girth Existence]\label{conj:HighGirthExistence}
Let $q>r\ge 2$ be integers. For every integer $g\geq 3$, every sufficiently large $K_q^r$-divisible complete $r$-uniform hypergraph admits a $K_q^r$-decomposition with girth at least $g$. 
\end{conj}

In 2022, the approximate version of the High Girth Existence Conjecture was settled by Delcourt and Postle \cite{DP22} and, independently, Glock, Joos, Kim, K{\"u}hn, and Lichev \cite{GJKKL24}; that is, they proved the existence of approximate $(n, q, r)$-Steiner systems of high girth with almost full size. In fact, both papers developed a general methodology that finds almost perfect matchings in hypergraphs that avoid a set of forbidden submatchings provided certain degree and codegree conditions are met. In particular, the general results then imply the approximate version of the High Girth Existence Conjecture as a corollary. Recently, in 2024 Delcourt and Postle~\cite{DPII} settled the High Girth Existence Conjecture in full by combining their forbidden submatching method with their newly developed refined absorption methodology and some additional ideas as follows.

\begin{thm}[Delcourt and Postle \cite{DPII}]\label{thm:DPHighGirth}
Conjecture~\ref{conj:HighGirthExistence} is true.
\end{thm}

\subsection{A Refined Absorption Proof of the High Girth Existence Conjecture.}

To prove the High Girth Existence Conjecture, we follow the same plan with modifications. Steps (1) and (3) remain the same, even using Lemma~\ref{lem:RandomX} and Lemma~\ref{lem:RegBoost} as before. For step (4), we apply the ``Forbidden Submatchings with Reserves'' Theorem (see~\cite{DPII}). Thus the remaining challenge is to modify step (2) as follows:
\begin{enumerate}
    \item[(2')] Construct a \emph{high girth} omni-absorber $A$ of $X$ that \emph{does not shrink `too many' configurations}. 
\end{enumerate}

The required definition of a high girth omni-absorber is straightforward as follows. 

\begin{definition}\label{def:OmniAbsorberCollectiveGirth}
A $K_q^r$-omni-absorber $A$ for an $r$-graph $X$ with decomposition function $\mathcal{Q}_A$ has \emph{collective girth at least $g$} if for every $K_q^r$-divisible subgraph $L$ of $X$, $\mathcal{Q}_A(L)$ has girth at least $g$. 
\end{definition}

An omni-absorber having high collective girth is not enough for our purposes to prove Theorem~\ref{thm:DPHighGirth}. The issue is that the various $\mathcal{Q}_A(L)$, individually or collectively, may forbid subconfigurations in $K_n^r\setminus (X\cup A)$ that would extend to forbidden configurations with $\mathcal{Q}_A(L)$. We need to ensure there are not too many such forbidden subconfigurations. Namely, the required assumption is whatever is necessary to apply the ``Forbidden Submatchings with Reserves'' Theorem. Again, to state the precise technical definition of ``not shrink too many'' requires many other definitions first and so we omit it. Thus there are two important proof parts then to enact (2'). \emph{First, how do we build an omni-absorber of collective girth at least $g$? Second, how do we ensure it does not shrink ``too many'' configurations?}

For the first part, we start with Theorem~\ref{thm:RefinedEfficientOmniAbsorber} which crucially provides a $C$-refined omni-absorber. Then we embed a private booster for each clique in the decomposition family. If each booster itself has the right high girth properties, we can use the Lov\'asz Local Lemma to ensure the resulting omni-absorber has collective girth at least $g$. The key notion is the rooted girth of a booster as follows.

\begin{definition}[Rooted Girth]
Let $q > r\ge 1$ be integers. Let $\mathcal{B}$ be a $K_q^r$-packing of a graph $G$. For $S\subseteq V(G)$, the \emph{rooted girth} of $\mathcal{B}$ at $S$ is the smallest integer $g\ge 1$ such that there exists a subset $\mathcal{B}'\subseteq \mathcal{B}$  with $|\mathcal{B}'|=g$ and $|V(\bigcup \mathcal{B}')\setminus S| < (q-r)\cdot g$. Similarly, if $B$ is a $K_q^r$-booster and $S\in \mathcal{B}_{\rm off}$, then we define the \emph{rooted girth} of $B$ at $S$ as the minimum of the girth of $\mathcal{B}_{{\rm on}}$, the girth of $\mathcal{B}_{{\rm off}}$, and the rooted girth of $\mathcal{B}_{{\rm on}}$ at $V(S)$.
\end{definition}

The first key technical result is the existence of rooted boosters of high rooted girth defined as follows.

\begin{thm}[High Rooted Girth Booster Theorem]\label{thm:HighRootedGirthBooster}
Let $q > r\ge 2$ and $g\ge 1$ be integers. If the ``High Cogirth Existence Theorem'' holds for $r':=r-1$, then there exists a $K_q^r$-booster $B$ with rooted girth at least $g$. 
\end{thm}

To construct such a high girth $K_q^r$-booster $B$ one uses the following construction: take a $K_{q-1}^{r-1}$-booster $B'$ and two new vertices $u,v$; we let $E(B):=\{ e\cup \{u\}, e\cup \{v\}: e\in B'\} \cup \bigcup_{Q\in \mathcal{B}_{\rm on}\cup \mathcal{B}_{\rm off}} \{ T\subseteq Q: |T|=r\}$. Thus the proof of Theorem~\ref{thm:HighRootedGirthBooster} uses induction on the uniformity, namely the existence of high girth $K_{q-1}^{r-1}$-boosters; however it is not enough to simply have these of high rooted girth. Rather, we need the existence of $K_{q-1}^{r-1}$-boosters of high \emph{cogirth} (see~\cite{DPII} for the definition). The only way we are able to construct these is to prove a stronger form of Theorem~\ref{thm:DPHighGirth} which asserts that there exist two disjoint high girth Steiner systems with high cogirth. The proof of that theorem follows similarly to that of Theorem~\ref{thm:DPHighGirth} with some minor modifications. We note this is still not the desired construction since if $S=e\cup \{u\}$ is the rooted clique, then the `mirror clique' $e\cup \{v\}$ may not have high rooted girth. So we need to layer new copies of the booster constructed above on top of the mirror clique itself until the final mirror clique is vertex-disjoint from the original rooted clique.

We also note though that being a high cogirth $K_{q-1}^{1}$-booster is equivalent to being the line graph $L(H)$ of a graph $H$ where $H$ is a $(q-1)$-regular bipartite graph of high girth (which were already known to exist). Now to prove the high girth omni-absorber theorem then, we add a private rooted booster of rooted girth at least $g$ for each clique in the decomposition family of the omni-absorber. 

Recall however there is the second difficulty of ensuring the resulting omni-absorber does not shrink too many configurations. This second part of not shrinking too many configurations is in fact handled by embedding the boosters randomly (where we used a sparsification trick to first randomly sparsify the set of potential boosters to ensure we do not shrink too many configurations and only then apply a high girth embedding lemma to embed from the sparsified sets edge-disjointly). 

\subsection{The ``Erd\H{o}s meets Nash-Williams'' Conjecture.}

In light of these two fundamental conjectures in extremal design theory, namely Nash-Williams' Conjecture and Erd\H{o}s' Conjecture, as well as the progress on them in recent years, it is natural to wonder about the minimum degree threshold to guarantee the existence of a high girth $K_3$-decomposition. In 2020, Glock, K\"uhn, and Osthus~\cite[Conjecture 7.7]{GKO20Survey} proposed the following combination of these two conjectures.

\begin{conj}[``Erd\H{o}s meets Nash-Williams''~\cite{GKO20Survey}]\label{conj:E-NW}
    For every integer $g$ and any sufficiently large $n$, every $K_3$-divisible graph $G$ on $n$ vertices with $\delta(G)\geq 3n/4$ admits a $K_3$-decomposition with girth at least $g$.
\end{conj}

The best progress to date is by Delcourt, Henderson, Lesgourgues, and Postle~\cite{DHLP25} who tied the existence of high girth $K_3$-decompositions to the fractional threshold $\delta^*_{K_3}$ using the method of refined absorption as follows.

\begin{thm}[Delcourt, Henderson, Lesgourgues, and Postle~\cite{DHLP25}]\label{thm:ENW} $\forall g\geq 3$ and $\varepsilon>0$, every sufficiently large $K_3$-divisible graph~$G$ on~$n$ vertices with 
    $\delta(G)\ge (\max\{\delta^*_{K_3},\frac{3}{4}\}+\varepsilon)n$
    admits a $K_3$-decomposition with girth at least $g$. 
\end{thm}

We now turn to the common generalization of Erd\H{o}s' Conjecture and the folklore generalization of Nash-Williams' Conjecture to $K_q$-decompositions. Delcourt, Henderson, Lesgourgues, and Postle~\cite{DHLP25} proved the following.

\begin{thm}[Delcourt, Henderson, Lesgourgues, and Postle~\cite{DHLP25}]\label{thm:GenralisationMain2q} $\forall g,q\geq 3$ and $\varepsilon>0$, every sufficiently large $K_q$-divisible graph~$G$ on~$n$ vertices with $\delta(G)\ge (\max\{\delta^*_{K_q},1-\frac{1}{2q-2}\}+\varepsilon)n$
    admits a $K_q$-decomposition with girth at least $g$. 
\end{thm}

Recall that the best known upper bound for $q \geq 4$ on $\delta^*_{K_q}$ is $1-\frac{1}{100q}$ by Montgomery~\cite{montgomery_fractional_2019}. Therefore the fractional bound is the bottleneck in the above theorem. However, it is entirely possible that the second term of $1-\frac{1}{2q-2}$ -- which is the threshold to embed absorbers via rooted degeneracy -- could be improved, namely to $1-\frac{1}{q+1}$ as in the result of Glock, K\"uhn, Lo, Montgomery, and Osthus~\cite{GKLMO19}. It is then natural to conjecture the following.
\begin{conj}\label{conj:ENWLargeCliquesNoEpsilon}
    For all integers $g\geq 3$ and $q\geq 3$ and real $\varepsilon > 0$, every sufficiently large $K_q$-divisible graph~$G$ on~$n$ vertices with $\delta(G)\ge (\delta^*_{K_q}+\varepsilon)n$ admits a $K_q$-decomposition with girth at least $g$.
\end{conj}

While most of the machinery of our proof would work for the potentially better bound of Conjecture~\ref{conj:ENWLargeCliquesNoEpsilon}, there are two areas which would require new ideas. One is utilizing the better absorber constructions from Glock, K\"{u}hn, Lo, Montgomery, and Osthus~\cite{GKLMO19} to embed the absorbers in a graph with the lower minimum degree condition; the other much harder issue is proving that girth boosters can also be embedded for such a lower minimum degree condition.

\subsection{A Refined Absorption Proof for Erd\H{o}s meets Nash-Williams.} As for modifying the refined absorption proof framework to prove Theorem~\ref{thm:ENW}, for step (1) we use Lemma~\ref{lem:RandomXHighMinDeg}. For step (2), we use Theorem~\ref{thm:NWRefinedEfficientOmniAbsorber} as the base and then embed girth boosters as in the high girth existence proof. Crucially, there are girth boosters for $K_3$-decompositions of rooted degeneracy at most $4$ and so these can be embedded in $G$ (and similarly girth boosters for $K_q$-decompositions of rooted degeneracy at most $2q-2$ for general $q$). For step (3), we start with Theorem~\ref{thm:LowWeightDecomposition} but we also have to avoid using \emph{dangerous triangles} (those that complete a low girth configuration with the high girth omni-absorber); so we use additional ideas such as seeding the decomposition so that every clique has some weight, removing dangerous triangles and then also using the Boosting Lemma of Glock, K\"uhn, Lo, Osthus~\cite{GKLO16} to fix it once more. For step (4), we use the Forbidden Submatchings with Reserves Theorem.

\section{Probabilistic Design Theory: Thresholds for Steiner Systems.}

\begin{definition}
    Let $q \geq 1$ be an integer.  For every integer $n \geq q$ and real $p \in [0,1]$, let $\mathcal{G}^{(q)}(n,p)$ be the random $n$-vertex $q$-uniform hypergraph in which each $q$-set is included as an edge independently with probability $p$.
    For an increasing property $\mathcal{P}$ of $q$-uniform hypergraphs and an increasing sequence $(n_i : i \in \mathbb N)$, a function $p^*=p^*(n)$ is a \textit{threshold} if as $i \rightarrow \infty$,
\begin{equation*}
\textrm{Prob}\left[ \mathcal{G}^{(q)}(n_i,p) \in \mathcal{P} \right]\rightarrow 
 \begin{cases}
      0 & \text{ if }  p=o(p^*) \text{ and }\\
      1 & \text { if } p=\omega(p^*).
    \end{cases} 
 \end{equation*}   
  Although thresholds are not unique, every pair of thresholds is related by some multiplicative constant factor, so we will refer to \textit{the} threshold for a property as is common in the literature.
\end{definition}

For $r = 1$, it is easy to see that $(n,q,1)$-Steiner systems exist if and only if $q \mid n$.  However, the ``threshold'' version of this problem -- determining the threshold for $\mathcal{G}^{(q)}(qn,p)$ to contain a $(qn,q,1)$-Steiner system, or equivalently, for $\mathcal{G}^{(q)}(qn,p)$ to contain a \textit{perfect matching} -- was considered to be one of the most important problems in probabilistic combinatorics until its resolution by Johansson, Kahn, and Vu~\cite{JKV08} in 2008. This problem, known as ``Shamir's Problem'' (as called by Erd\H{o}s \cite{E81b}), was also a major motivation behind the Kahn--Kalai Conjecture~\cite{KK07}, recently proved by Park and Pham~\cite{PP23}.  (The ``fractional'' version of this conjecture posed by Talagrand~\cite{Ta10} was proved slightly earlier by Frankston, Kahn, Narayanan, and Park~\cite{FKNP21}.)  The $q = 2$ case, concerning perfect matchings in random graphs, is a classic result of Erd\H{o}s and R\'enyi~\cite{ER66}.  Standard probabilistic arguments show that $n^{-q+1}\log n$ is the threshold for the property that every vertex of $\mathcal{G}^{(q)}(qn,p)$ is contained in at least one edge, which in turn implies that the $(n,q,1)$-Steiner system threshold is at least $n^{-q + 1}\log n$.  Johansson, Kahn, and Vu~\cite{JKV08} proved a matching upper bound, and this fact also follows from the Park--Pham Theorem~\cite{PP23}. 

For $q > r > 1$, determining the threshold for $\mathcal{G}^{(q)}(n, p)$ to contain an $(n,q,r)$-Steiner system (for $n$ satisfying the divisibility conditions) is more challenging, as any non-trivial upper bound on the threshold necessarily implies not only that these designs exist, but that they exist ``robustly''. Hence, any result of this type seemed out of reach until Keevash's~\cite{K14} breakthrough in 2014. In 2017, Simkin~\cite{S17} conjectured that $n^{-1}\log n$ is the threshold for $\mathcal{G}^{(q)}(n,p)$ to contain an $(n,q,q-1)$-Steiner system for every integer $q > 1$, and Keevash posed the same problem for $q = 3$ (the case of Steiner triple systems) in his 2018 ICM talk. 
Similar to the $r = 1$ case, standard arguments show that $n^{-1}\log n$ is a lower bound on this threshold.
Utilizing the work of Keevash~\cite{K14} on the Existence Conjecture, Simkin~\cite{S17} showed that the threshold for $(n,q,q-1)$-Steiner systems for any fixed integer $q > 1$ is at most $n^{-\varepsilon}$ for some $\varepsilon > 0$ that depends on $q$.
Following the breakthroughs of Frankston, Kahn, Narayanan, and Park~\cite{FKNP21} as well as Park and Pham~\cite{PP23} on the Kahn--Kalai Conjecture~\cite{KK07}, a series of results in 2022 culminated in the resolution of the threshold problem for Steiner triple systems.
First, Sah, Sawhney, and Simkin~\cite{SSS23} proved an upper bound on the threshold of $n^{-1+o(1)}$; Kang, Kelly, K\"uhn, Methuku, and Osthus~\cite{KKKMO22} proved the better bound of $n^{-1}\log^2 n$. Subsequently Jain and Pham~\cite{JP22} and independently Keevash~\cite{K22} settled the problem, proving that $n^{-1}\log n$ upper bounds the threshold.

Besides the cases of $r = 1$ and of $r = 2$ and $q = 3$, the threshold problem for $(n, q, r)$-Steiner systems is still open.  
Kang, Kelly, K\"uhn, Methuku, and Osthus~\cite{KKKMO22} conjectured that the threshold is $n^{-q+r}\log n$.  
For $p = o(n^{-q+r}\log n)$, asymptotically almost surely $\mathcal{G}^{(q)}(n,p)$ will contain an $r$-set of vertices that is not contained in any edge and thus does not contain an $(n,q,r)$-Steiner system.  Hence, to prove this conjecture it suffices to show that $n^{-q+r}\log n$ upper bounds the threshold for $\mathcal{G}^{(q)}(n,p)$ to contain an $(n,q,r)$-Steiner system, as follows.

\begin{conj}[Kang, Kelly, K\"uhn, Methuku, and Osthus \cite{KKKMO22}]\label{conj:KKKMO} $\forall~q > r$: If $\binom{q - i}{r - i} \mid \binom{n - i}{r - i}$ for all $i \in \{0, \dots, r-1\}$ and $p = \omega(n^{-q+r}\log n)$, then
a.a.s. $\mathcal{G}^{(q)}(n, p)$ contains an $(n,q,r)$-Steiner system.
\end{conj}

The main result of~\cite{DKPIV} is progress on the $r=2$ case of Conjecture~\ref{conj:KKKMO} as follows.

\begin{thm}[Delcourt, Kelly, and Postle~\cite{DKPIV}]\label{thm:ExistenceSpread} $\forall q > 2$ the following holds:  
If $q - 1 \mid n - 1$ and $\binom{q}{2}\mid \binom{n}{2}$ and $p\geq n^{-(q-6)/2}$, then a.a.s. $\mathcal{G}^{(q)}(n,p)$ contains an $(n, q, 2)$-Steiner system.
\end{thm}

Besides the $q = 3$ case, no non-trivial upper bound on the threshold for $(n,q,2)$-Steiner systems was known prior to the above theorem.  
The exponent of $n$ in the lower bound for $p$ in Theorem~\ref{thm:ExistenceSpread} is essentially a factor of two away from the conjectured value.
For $p = 1$, Theorem~\ref{thm:ExistenceSpread} is equivalent to Wilson's~\cite{WI, WII, WIII} result that the Existence Conjecture holds for $r = 2$. To prove Theorem~\ref{thm:ExistenceSpread}, one employs the Park--Pham Theorem~\cite{PP23}.  In fact, the fractional version of Frankston, Kahn, Narayanan, and Park~\cite{FKNP21} suffices.  To use this theorem, our main objective is to show the existence of a sufficiently ``spread'' probability distribution on the $K_q$-decompositions of $K_n$ (provided $K_n$ is $K_q$-divisible). In this context, the definition of \textit{spread} is as follows. %\textbf{Fix probability symbol below}

\begin{definition}
    A probability distribution on $K_q$-decompositions $\mathcal{H}$ of a graph $G$ is \textit{$\sigma$-spread} if $\text{Prob}[\mathcal{S} \subseteq \mathcal{H}] \leq \sigma^{|\mathcal{S}|}$  for all $K_q$-packings $\mathcal{S}$ of $G$.
\end{definition}

The Park--Pham Theorem implies that for some absolute constant $K > 0$, if there exists a $\sigma_n$-spread probability distribution on $K_q$-decompositions of $K_n$ and $p \geq K\sigma_n\log n$, then $\mathcal{G}^{(q)}(n, p)$ asymptotically almost surely contains an $(n, q, 2)$-Steiner system.  Hence, Theorem~\ref{thm:ExistenceSpread} follows immediately from the following result.

\begin{thm}[Delcourt, Kelly, and Postle~\cite{DKPIV}]\label{thm:spread-decomposition}
$\forall~q > 2$, $\exists~\beta > 0$ s.t.~for all $n$ large enough: If $K_n$ is $K_q$-divisible, then there exists an $\left(n^{-(q - 6)/2 - \beta}\right)$-spread probability distribution on $K_q$-decompositions of $K_n$.
\end{thm}

We also note that Delcourt, Kelly, and Postle~\cite{DKPIII} use a similar (but much more technical) refined absorption approach to make progress on conjectures of Yuster on the threshold for when the random graph $G(n,p)$ has a $K_3$-packing (or more generally $K_q$-packing) with optimal leave and similarly for when the random regular graph $G_{n,d}$ admits a $K_q$-decomposition provided it is $K_q$-divisible.  

\subsection{A Refined Absorption Proof of a fairly Spread Distribution on Steiner Systems.}

We now discuss how to adapt the refined absorption proof framework to the spread setting. The main difficulty for the proof of Theorem~\ref{thm:spread-decomposition} lies in step (2), namely we need to embed an omni-absorber in a spread way for Theorem~\ref{thm:spread-decomposition}. Similarly, we require a spread version of nibble with reserves for step (4). However, since we are not striving for the optimal value of $p$, this may be accomplished via a simple random sparsification argument which will only require a $p$ that is ${\rm polylog}~n$ times the conjectured value of $p$ (and hence is not the bottleneck in the proof).

For step (2), we once more use Theorem~\ref{thm:RefinedEfficientOmniAbsorber} as a black box and then embed certain boosters, specifically ones with good randomness properties we call \textit{spread boosters}, to boost the spreadness. (Similar to how we boosted the girth via \emph{girth boosters} for the high girth problem). Indeed, for the construction of these boosters is suffices to take $L(K_{q-1},K_{q-1})$ joined completely to two new vertices -- except in that simple construction the mirror clique will not have the desired randomness properties; so we once more layer these boosters until the mirror clique is vertex-disjoint from the rooted clique. Returning to the proof then with spread boosters whose decompositions are sufficiently low density, we argue that the boosters are spread if chosen independently. However, we must use the Lov\'asz Local Lemma to ensure that the boosters are embedded disjointly. So finally we argue via the Lov\'asz Local Lemma distribution (see~\cite{JP22} for a description) that the spreadness of a fixed set of cliques is not too much larger for the disjoint case than if the choices were truly independent. 

\section{Future Directions and Open Questions.}

As for future directions, there are four main directions. 

\begin{problem}
Determine the minimum degree thresholds for $K_q$-decompositions and $K_q^r$-decompositions more generally.   
\end{problem}

\begin{problem}
Determine the probability thresholds for $K_q$-decompositions and $K_q^r$-decompositions more generally.   
\end{problem}

\begin{problem}
Extend refined absorption to partite and partitioned hypergraphs (and more generally to all of Keevash's Designs II~\cite{K18II} settings).
\end{problem}

\begin{problem}
Unify the various problems (minimum degree, high girth, spread, partite, etc.) into one super theorem (possibly with tradeoffs). 
\end{problem}

Indeed, I believe so much in this goal and in light of the evidence from Erd\H{o}s meets Nash-Williams and the method of refined absorption, I am willing to make this a postulate of extremal and probabilistic design theory going forwards. Indeed, perhaps all the mountains of design theory are really just sides of the same mountain.

\begin{postulate}[Postle]
There exists one unified theorem of designs.    
\end{postulate}

\section*{Acknowledgments.} The author would like to thank Michelle Delcourt for many helpful discussions. The author would also like to thank Cicely Henderson, Richard Lang, and Thomas Lesgourgues for helpful comments. 

\bibliographystyle{siamplain}
\bibliography{bibliography}

\begin{thebibliography}{10}

\bibitem{AKS81}
{\sc M.~Ajtai, J.~Koml{\'o}s, and E.~Szemer{\'e}di}, {\em A dense infinite {S}idon sequence}, European Journal of Combinatorics, 2 (1981), pp.~1--11.

\bibitem{AY05}
{\sc N.~Alon and R.~Yuster}, {\em On a hypergraph matching problem}, Graphs and Combinatorics, 21 (2005), pp.~377--384.

\bibitem{BKLMO17}
{\sc B.~Barber, D.~K{\"u}hn, A.~Lo, R.~Montgomery, and D.~Osthus}, {\em Fractional clique decompositions of dense graphs and hypergraphs}, Journal of Combinatorial Theory, Series B, 127 (2017), pp.~148--186.

\bibitem{BKLO16}
{\sc B.~Barber, D.~K{\"u}hn, A.~Lo, and D.~Osthus}, {\em Edge-decompositions of graphs with high minimum degree}, Advances in Mathematics, 288 (2016), pp.~337--385.

\bibitem{BW19}
{\sc T.~Bohman and L.~Warnke}, {\em Large girth approximate {S}teiner triple systems}, Journal of the London Mathematical Society, 100 (2019), pp.~895--913.

\bibitem{BES73}
{\sc W.~Brown, P.~Erd{\H{o}}s, and V.~S{\'o}s}, {\em Some extremal problems on $r$-graphs}, New directions in the theory of graphs ({P}roceedings of the {T}hird {A}nn {A}rbor {C}onference, {U}niv. {M}ichigan, {A}nn {A}rbor, {M}ich., 1971),  (1973), pp.~53--63.

\bibitem{CKKO19}
{\sc P.~Condon, J.~Kim, D.~K{\"u}hn, and D.~Osthus}, {\em A bandwidth theorem for approximate decompositions}, Proceedings of the London Mathematical Society, 118 (2019), pp.~1393--1449.

\bibitem{CH63}
{\sc K.~Corr{\'a}di and A.~Hajnal}, {\em On the maximal number of independent circuits in a graph}, Acta Mathematica Hungarica, 14 (1963), pp.~423--439.

\bibitem{DHLP25b}
{\sc M.~Delcourt, C.~Henderson, T.~Lesgourgues, and L.~Postle}, {\em Beyond {N}ash–{W}illiams: {C}ounterexamples to clique decomposition thresholds for all cliques larger than triangles}, arXiv:2508.20819,  (2025).

\bibitem{DHLP25}
{\sc M.~Delcourt, C.~Henderson, T.~Lesgourgues, and L.~Postle}, {\em Erd{\H{o}}s meets {N}ash-{W}illiams}, arXiv:2507.23624,  (2025).

\bibitem{DKPIII}
{\sc M.~Delcourt, T.~Kelly, and L.~Postle}, {\em Clique {D}ecompositions in {R}andom {G}raphs via {R}efined {A}bsorption}, arXiv:2402.17857,  (2024).

\bibitem{DKP24}
{\sc M.~Delcourt, T.~Kelly, and L.~Postle}, {\em A short proof of the existence of ${K}_q^r$-absorbers}, arXiv:2412.09710,  (2024).

\bibitem{DKPIV}
{\sc M.~Delcourt, T.~Kelly, and L.~Postle}, {\em Thresholds for $(n,q,2)$-{S}teiner {S}ystems via {R}efined {A}bsorption}, arXiv:2402.17858,  (2024).

\bibitem{DP2021progress}
{\sc M.~Delcourt and L.~Postle}, {\em Progress towards {N}ash-{W}illiams' {C}onjecture on triangle decompositions}, Journal of Combinatorial Theory, Series B, 146 (2021), pp.~382--416.

\bibitem{DP22}
{\sc M.~Delcourt and L.~Postle}, {\em Finding an almost perfect matching in a hypergraph avoiding forbidden submatchings}, arXiv:2204.08981,  (2022).

\bibitem{DP22BES}
{\sc M.~Delcourt and L.~Postle}, {\em The limit in the $(k+2,k)$-{P}roblem of {B}rown, {E}rd{\H{o}}s and {S}{\'o}s exists for all $k\geq 2$}, Proceedings of the American Mathematical Society, 152 (2024), pp.~1881--1891.

\bibitem{DPII}
{\sc M.~Delcourt and L.~Postle}, {\em Proof of the {H}igh {G}irth {E}xistence {C}onjecture via {R}efined {A}bsorption}, arXiv:2402.17856,  (2024).

\bibitem{DPI}
{\sc M.~Delcourt and L.~Postle}, {\em {R}efined {A}bsorption: {A} {N}ew {P}roof of the {E}xistence {C}onjecture}, arXiv:2402.17855,  (2024).

\bibitem{dross2016fractional}
{\sc F.~Dross}, {\em Fractional triangle decompositions in graphs with large minimum degree}, SIAM Journal on Discrete Mathematics, 30 (2016), pp.~36--42.

\bibitem{D12}
{\sc P.~Dukes}, {\em Rational decomposition of dense hypergraphs and some related eigenvalue estimates}, Linear Algebra and its Applications, 436 (2012), pp.~3736--3746.

\bibitem{dukes_minimum_2020}
{\sc P.~J. Dukes and D.~Horsley}, {\em On the minimum degree required for a triangle decomposition}, SIAM Journal on Discrete Mathematics, 34 (2020), pp.~597--610.

\bibitem{E73}
{\sc P.~Erd{\H{o}}s}, {\em Problems and results in combinatorial analysis}, in Colloq. Internat. Theor. Combin. Rome, 1973, pp.~3--17.

\bibitem{E81b}
{\sc P.~Erd{\H{o}}s}, {\em On the combinatorial problems which {I} would most like to see solved}, Combinatorica, 1 (1981), pp.~25--42.

\bibitem{EH63}
{\sc P.~Erd{\H{o}}s and H.~Hanani}, {\em On a limit theorem in combinatorial analysis}, Publ. Math. Debrecen, 10 (1963), pp.~10--13.

\bibitem{ER66}
{\sc P.~Erd{\H{o}}s and A.~R{\'e}nyi}, {\em On the existence of a factor of degree one of a connected random graph}, Acta Math. Acad. Sci. Hungar, 17 (1966), p.~192.

\bibitem{FR85}
{\sc P.~Frankl and V.~R{\"o}dl}, {\em Near perfect coverings in graphs and hypergraphs}, European Journal of Combinatorics, 6 (1985), pp.~317--326.

\bibitem{FKNP21}
{\sc K.~Frankston, J.~Kahn, B.~Narayanan, and J.~Park}, {\em Thresholds versus fractional expectation-thresholds}, Annals of Mathematics, 194 (2021), pp.~475--495.

\bibitem{FR13}
{\sc Z.~F{\"u}redi and M.~Ruszink{\'o}}, {\em Uniform hypergraphs containing no grids}, Advances in Mathematics, 240 (2013), pp.~302--324.

\bibitem{garaschuk2014linear}
{\sc K.~Garaschuk}, {\em Linear methods for rational triangle decompositions}, PhD thesis, University of {V}ictoria, 2014.

\bibitem{GJKKL24}
{\sc S.~Glock, F.~Joos, J.~Kim, M.~K{\"u}hn, and L.~Lichev}, {\em Conflict-free hypergraph matchings}, Journal of the London Mathematical Society, 109 (2024), p.~e12899.

\bibitem{GKLMO19}
{\sc S.~Glock, D.~K{\"u}hn, A.~Lo, R.~Montgomery, and D.~Osthus}, {\em On the decomposition threshold of a given graph}, Journal of Combinatorial Theory, Series B, 139 (2019), pp.~47--127.

\bibitem{GKLO20}
{\sc S.~Glock, D.~K{\"u}hn, A.~Lo, and D.~Osthus}, {\em On a conjecture of {E}rd{\H{o}}s on locally sparse {S}teiner triple systems}, Combinatorica, 40 (2020), pp.~363--403.

\bibitem{GKLO16}
{\sc S.~Glock, D.~K{\"u}hn, A.~Lo, and D.~Osthus}, {\em The existence of designs via iterative absorption: {H}ypergraph ${F}$-designs for arbitrary ${F}$}, Memoirs of the American Mathematical Society, 284 (2023).

\bibitem{GKO20Survey}
{\sc S.~Glock, D.~Kühn, and D.~Osthus}, {\em Extremal aspects of graph and hypergraph decomposition problems}, London Mathematical Society Lecture Note Series, Cambridge University Press, 2021, p.~235–266.

\bibitem{GJ73}
{\sc J.~E. Graver and W.~Jurkat}, {\em The module structure of integral designs}, Journal of Combinatorial Theory, Series A, 15 (1973), pp.~75--90.

\bibitem{gustavsson1991decompositions}
{\sc T.~Gustavsson}, {\em Decompositions of large graphs and digraphs with high minimum degree}, PhD thesis, University of Stockholm, 1991.

\bibitem{HS70}
{\sc A.~Hajnal and E.~Szemer\'edi}, {\em Proof of a conjecture of {P}. {E}rd{\H{o}}s}, Combinatorial Theory and its Applications, II (1970), pp.~601--603.

\bibitem{JP22}
{\sc V.~Jain and H.~T. Pham}, {\em Optimal thresholds for {L}atin squares, {S}teiner triple systems, and edge colorings}, arXiv:2212.06109,  (2022).

\bibitem{JKV08}
{\sc A.~Johansson, J.~Kahn, and V.~Vu}, {\em Factors in random graphs}, Random Structures Algorithms, 33 (2008), pp.~1--28.

\bibitem{K96}
{\sc J.~Kahn}, {\em Asymptotically good list-colorings}, Journal of Combinatorial Theory, Series A, 73 (1996), pp.~1--59.

\bibitem{KK07}
{\sc J.~Kahn and G.~Kalai}, {\em Thresholds and expectation thresholds}, Combin. Probab. Comput., 16 (2007), pp.~495--502.

\bibitem{KKKMO21}
{\sc D.~Kang, T.~Kelly, D.~K\"{u}hn, A.~Methuku, and D.~Osthus}, {\em Graph and hypergraph colouring via nibble methods: {A} survey}, European Congress of Mathematics, 07 2023, pp.~771--823.

\bibitem{KKKMO22}
{\sc D.~Y. Kang, T.~Kelly, D.~K{\"u}hn, A.~Methuku, and D.~Osthus}, {\em Thresholds for {L}atin squares and {S}teiner triple systems: {B}ounds within a logarithmic factor}, Transactions of the American Mathematical Society, 376 (2023), pp.~6623--6662.

\bibitem{K14}
{\sc P.~Keevash}, {\em The existence of designs}, arXiv preprint arXiv:1401.3665,  (2014).

\bibitem{K18II}
{\sc P.~Keevash}, {\em The existence of designs {I}{I}}, arXiv:1802.05900,  (2018).

\bibitem{K22}
{\sc P.~Keevash}, {\em The optimal edge-colouring threshold}, arXiv:2212.04397,  (2022).

\bibitem{K24}
{\sc P.~Keevash}, {\em A short proof of the existence of designs}, arXiv:2411.18291,  (2024).

\bibitem{KL20}
{\sc P.~Keevash and J.~Long}, {\em The {B}rown-{E}rd{\H{o}}s-{S}{\'o}s conjecture for hypergraphs of large uniformity}, arXiv:2007.14824,  (2020).

\bibitem{K47}
{\sc T.~P. Kirkman}, {\em On a problem in combinatorics}, Cambridge Dublin Math. J, 2 (1847), pp.~191--204.

\bibitem{KKO15}
{\sc F.~Knox, D.~K\"{u}hn, and D.~Osthus}, {\em Edge-disjoint {H}amilton cycles in random graphs}, Random Structures Algorithms, 46 (2015), pp.~397--445.

\bibitem{KO13}
{\sc D.~K\"{u}hn and D.~Osthus}, {\em Hamilton decompositions of regular expanders: a proof of {K}elly's conjecture for large tournaments}, Advances in Mathematics, 237 (2013), pp.~62--146.

\bibitem{KSSS2024STS}
{\sc M.~Kwan, A.~Sah, M.~Sawhney, and M.~Simkin}, {\em High-girth {S}teiner triple systems}, Annals of Mathematics, 200 (2024), pp.~1059--1156.

\bibitem{lang2023tiling}
{\sc R.~Lang}, {\em Tiling dense hypergraphs}, arXiv:2308.12281,  (2023).

\bibitem{LPR93}
{\sc H.~Lefmann, K.~T. Phelps, and V.~R{\"o}dl}, {\em Extremal problems for triple systems}, Journal of Combinatorial Designs, 1 (1993), pp.~379--394.

\bibitem{mantel}
{\sc W.~Mantel}, {\em Problem 28 (solution by {Gouwentak, Mantel, Teixeira de Mattes, Schuh and Wythoff})}, Wiskundige Opgaven, 10 (1907), pp.~60--61.

\bibitem{montgomery_fractional_2019}
{\sc R.~Montgomery}, {\em Fractional clique decompositions of dense graphs}, Random Structures \& Algorithms, 54 (2019), pp.~779--796.

\bibitem{M19b}
{\sc R.~Montgomery}, {\em Spanning trees in random graphs}, Advances in Mathematics, 356 (2019), p.~106793.

\bibitem{NW70}
{\sc C.~S.~J.~A. Nash-Williams}, {\em An unsolved problem concerning decomposition of graphs into triangles}, Combinatorial Theory and its Applications, 3 (1970), pp.~1179--1183.

\bibitem{PP23}
{\sc J.~Park and H.~Pham}, {\em A proof of the {K}ahn-{K}alai conjecture}, J. Amer. Math. Soc., 37 (2024), pp.~235--243.

\bibitem{PS89}
{\sc N.~Pippenger and J.~Spencer}, {\em Asymptotic behavior of the chromatic index for hypergraphs}, Journal of Combinatorial Theory, Series A, 51 (1989), pp.~24--42.

\bibitem{R85}
{\sc V.~R{\"o}dl}, {\em On a packing and covering problem}, European Journal of Combinatorics, 6 (1985), pp.~69--78.

\bibitem{RRS06}
{\sc V.~R{\"o}dl, A.~Ruci{\'n}ski, and E.~Szemer{\'e}di}, {\em A {D}irac-type theorem for 3-uniform hypergraphs}, Combinatorics, Probability and Computing, 15 (2006), pp.~229--251.

\bibitem{SSS23}
{\sc A.~Sah, M.~Sawhney, and M.~Simkin}, {\em Threshold for {S}teiner triple systems}, Geometric and Functional Analysis,  (2023), pp.~1--32.

\bibitem{S17}
{\sc M.~Simkin}, {\em (n, k, k-1)-{S}teiner systems in random hypergraphs}, arXiv:1711.01975,  (2017).

\bibitem{Ta10}
{\sc M.~Talagrand}, {\em Are many small sets explicitly small?}, in Proceedings of the forty-second ACM symposium on Theory of computing, 2010, pp.~13--36.

\bibitem{turan1941extremal}
{\sc P.~Tur{\'a}n}, {\em Eine {Extremalaufgabe} aus der {Graphentheorie}}, Matematikai és. Fizikai Lapok, 48 (1941), pp.~436--452.

\bibitem{WI}
{\sc R.~M. Wilson}, {\em An existence theory for pairwise balanced designs {I}. {C}omposition theorems and morphisms}, Journal of Combinatorial Theory, Series A, 13 (1972), pp.~220--245.

\bibitem{WII}
{\sc R.~M. Wilson}, {\em An existence theory for pairwise balanced designs {II}. {T}he structure of {PBD}-closed sets and the existence conjectures}, Journal of Combinatorial Theory, Series A, 13 (1972), pp.~246--273.

\bibitem{W73}
{\sc R.~M. Wilson}, {\em The necessary conditions for $t$-designs are sufficient for something}, Utilitas Math, 4 (1973), pp.~207--215.

\bibitem{WIII}
{\sc R.~M. Wilson}, {\em An existence theory for pairwise balanced designs, {III}: {P}roof of the existence conjectures}, Journal of Combinatorial Theory, Series A, 18 (1975), pp.~71--79.

\bibitem{Yuster2005asymptotically}
{\sc R.~Yuster}, {\em Asymptotically optimal ${K}_k$-packings of dense graphs via fractional ${K}_k$-decompositions}, Journal of Combinatorial Theory, Series B, 95 (2005), pp.~1--11.

\end{thebibliography}
\end{document}